\documentclass[10pt,reqno]{amsart}

\usepackage{amscd,amssymb,amsmath,amsthm,mathrsfs,dsfont} 
\usepackage{calc}
\usepackage{a4wide}
\usepackage{graphicx}
\usepackage{color}
\usepackage[]{hyperref}
\usepackage{bbm}
\usepackage[lofdepth,lotdepth]{subfig}
\usepackage{tikz} 
\usetikzlibrary{positioning}
\usepackage[shortlabels]{enumitem}
\usepackage[backend=biber,style=ieee,giveninits=true,sorting=anyt]{biblatex}
\newcommand{\vardbtilde}[1]{\tilde{\raisebox{0pt}[0.85\height]{$\tilde{#1}$}}}
\newtheorem{thm}{Theorem}
\newtheorem{defn}{Definition}
\newtheorem{lemma}{Lemma}
\newtheorem{pro}{Proposition}
\newtheorem{rk}{Remark}

\newtheorem{assumption}{Assumption}

 \hypersetup{
    colorlinks=true,
    linkcolor=blue,
    filecolor=blue,      
    urlcolor=blue,
    citecolor=blue,
}

\numberwithin{equation}{section} 

\DeclareRobustCommand*{\ora}{\overrightarrow}
\DeclareRobustCommand*{\ola}{\overleftarrow}

\def\R{\mathbb{R}}

\def\L{\Lambda}

\def\N{\mathbb{N}}
\def\Z{\mathbb{Z}}

\def \L {\Lambda}

\begin{document}
\title[Height-offset variables on trees]{Height-offset variables and pinning at infinity for gradient Gibbs measures on trees}

\author{Florian Henning and Christof K\"ulske}
\address{Florian Henning \\ Department of Mathematics \& Computer Science,
TU Eindhoven,\,
5600MB Eindhoven,
The Netherlands.\\
\url{https://orcid.org/0000-0001-5384-2832}}
\email{F.B.Henning@TUE.nl}

\address{Christof\ K\"ulske\\ Fakult\"at f\"ur Mathematik,
Ruhr University Bochum, Postfach 102148,\,
44721 Bochum,
Germany.\\
\url{https://orcid.org/0000-0001-9975-8329}}
\email {Christof.Kuelske@ruhr-uni-bochum.de}

\begin{abstract} 
Height-offset variables (HOVs) provide a mechanism, known as "pinning at infinity," to lift 
gradient Gibbs measures (GGMs)—describing interface increments—to proper Gibbs measures that describe absolute heights. 
Starting from Sheffield's seminal framework, we study HOVs for nearest-neighbor integer-valued gradient models on regular trees, under broad classes of transfer operators requiring only finite second moments and without assuming convexity. 
We first establish the existence of HOVs as martingale limits, prove the infinite differentiability of their Lebesgue densities, and demonstrate exponential concentration for the associated pinned Gibbs measures. 
Next we uncover a fundamental trade-off,   
as the Gibbs measures arising by “pinning at infinity” paradoxically lose several desirable structural properties. 
We rigorously show that they lose tree-automorphism invariance, the tree-indexed Markov chain property, and extremality within the class of Gibbs measures. 
Our analysis relies on martingale theory, novel past- and future-tail decompositions, and infinite product representations for moment generating functions, and it applies to free GGMs, as well as to GGMs of height-period two. 

\end{abstract}
\maketitle



\hypersetup{linkcolor=black}
\tableofcontents
\hypersetup{linkcolor=blue}
\section{Introduction}
The notion of a gradient Gibbs measure (GGM) was introduced by Funaki and Spohn in \cite{FS97} for the Ginzburg-Landau interface model to classify the ergodic measures for the shape of an interface in space which is subject to a stochastic dynamics.
As opposed to Gibbs measures (GM), which describe the absolute heights of the interface, gradient Gibbs measures describe the heights of the interface \textit{modulo a joint shift} in the height-direction. This can be represented by the field of increments (\textit{gradients}) along the oriented edges of the respective graph. 

For background on gradient models, on various graphs, concerning various aspects of localization and delocalization properties,  see 
\cite{CoKuLN23,CKu15,DaHaPe23,EnKu08,Ve06,DeuGiIo00,La22,LaOt23,Mi11,Sa24,Se24,DeRo24,CKu15,DuHaLaRaRa22}.
In particular, \cite{DaHaPe23, EnKu08,DeuGiIo00, DeRo24, FS97, Sh05,LaOt23} specifically consider gradient models on the lattice $\Z^d$ or the lattice torus $(\Z/nZ)^d$. On the other hand, \cite{La22} considers planar graphs of maximal degree three which are invariant under the action of some distinguished lattice $\mathcal{L}$ and provides two distinct sufficient assumptions on the interaction that guarantee that there are no (spatially) $\mathcal{L}$-invariant proper Gibbs measures. \cite{Se24} studies the Ginzburg-Landau-model on locally finite connected graphs which are \textit{percolation-invariant} in that for some $p<1$ the $p$-Bernoulli bond percolation on the respective graph almost surely possesses a transient infinite cluster. Finally, for results on the tree see \cite{LaTo23,CoKuLN23}.

On the lattice $\Z^d$ the theory was significantly enhanced by Sheffield \cite{Sh05}. Amongst other results, for continuous spin variables Sheffield uniquely characterizes \textit{liftable}  
(or as he calls them \textit{smooth}) gradient phases by their average heights modulo 1 and their height-offset spectrum. 
Here, an infinite-volume gradient state is liftable, if it arises as the projection 
of a proper Gibbs measure to the gradient variables. The height-offset spectrum 
is defined as the distribution of the underlying \textit{height-offset variable} (HOV).  

Most importantly, a height-offset variable is by definition  a tail-measurable observable on the infinite-volume state space 
of total height configurations  $\Omega_0^V$ which enjoys covariance under a joint shift of the configuration in the height-direction (cf. Definition \ref{def: HeightOffset} below).
Height-offset variables make rigorous the idea of \textit{pinning a gradient configuration at infinity} 
by assigning to it an absolute height. 
They can be used to lift gradient Gibbs measures to proper Gibbs measures, which describe absolute heights at the vertices of the graph. 
This may seem to be a contradiction at first, and there is in fact an interesting general theory behind, as we will see. 
To our knowledge the phenomenon was first shown to appear with a concrete example of a model on a tree, 
in the inspiring  paper \cite{LaTo23}, where Lammers and Toninelli discuss how to interpret the free gradient state of the 
graph homomorphism model (cf. also e.g., \cite{ChPeShTa21} for the same model on the lattice) on regular trees as a proper Gibbs measure, using the existence of a non-trivial height-offset variable. 

Their model is prototypical and very relevant, 
as it is based on symmetric nearest neighbors jumps, 
without parameter. 
However, it remained unclear, 
how far the phenomenon extends when 
one looks at general gradient interactions, 
assuming only sufficient summability. 
Are there always HOVs, and if so, what regularity properties do they have?  
What are the possible consequences, are there pinned measures, what is their structure? What can we say about extremality, Markov chain property, translation invariance? 
This begs for a theory for 
HOVs for more general gradient Gibbs measures on regular trees with more general classes of interactions. This is the motivation 
for our present work.  
We restrict ourselves to regular trees in this paper, leaving more general graphs as supporting spaces for the future.

We recall that a fixed gradient interaction on trees, may admit a large variety of consistent 
infinite-volume gradient Gibbs measures \cite{HKLR19}. 
We will focus in our study on two types 
of GGMs, which are consistent measures 
for the same fixed interaction described by
the transfer operator.

The most prominent and popular type of GGMs for which our results hold, are \textit{free states}, obtained with open boundary conditions. 
While all increments are independent under 
the measure, the liftability is non-trivial, 
and the corresponding lifted measure 
shows the lack of translation-invariance 
and extremality, as we will show. 

Moreover, we include in our analysis, at not too much more cost in the proofs, also the case of the more complex period-$2$ \textit{height-periodic} GGMs.  
These are GGMs for the same interaction, 
but they have an interesting dependence structure, 
in which the gradient variables are resampled conditional 
on a suitably chosen hidden 
Ising model, see Section \ref{subsec: results 2-height periodic}. 
This shows that 
our methods are not restricted to the free measures.

\vspace{5pt}
\paragraph{\textbf{Main results 1: Existence of HOVs and peculiarities of pinned measures}} 
In the present note we consider free gradient states (Theorem \ref{thm: Free measure}) and height-periodic measures (Theorem \ref{thm: 2periodic}) asking for the possibility of pinning at infinity. 
Specifically we will explain how to obtain in complete generality the
existence of height-offset variables for free GGMs and GGMs of period $2$. 
To show existence of HOVs, we prove that limits 
of empirical height sums over spheres 
(after an arbitrary auxiliary pinning at a given site) converge a.s. and in $L^2$ to a height-offset variable $H$. 
This works for the free gradient measure, see Theorem \hyperlink{res: existence free}{2a)}, but also for
height-period $2$ gradient measures, see Theorem \hyperlink{res: existence 2per}{3a)}, under an $L^2$-condition on the defining 
a-priori increment interaction along any edge described by \hyperlink{par: transfer operator}{the transfer operator} $Q$. 

For convenience of the reader, and to be on safe grounds to develop the theory, 
we include a detailed proof of 
how to pin at infinity on trees by means of an HOV and obtain a GM from the GGM, see Theorem \ref{thm: DLR}. 
From here we derive consequences in our paper, both from a structural probabilistic point of view, and 
from a quantitative analytical point of view. 
Namely, this pinned-at-infinity measure however, always has 
peculiar properties: 
It necessarily loses 
translation invariance (tree-automorphism invariance resp.), see Theorems \hyperlink{res: peculiarity 1 free}{2c1)} and \hyperlink{res: peculiarity 1 2per}{3c1)}.  
Note that in our paper we call equivalently a measure 
\textit{homogeneous} if it 
is invariant under all automorphisms of the 
tree. Recall that the translations of the tree 
when the tree is viewed as a group, form a subclass of the tree-automorphisms.

Moreover, we show that the pinned-at-infinity measure loses 
extremality in the simplex of all Gibbs measures, see Theorem \hyperlink{res: peculiarity 2 2per}{3c2)}, and in case of the free state even more the tree-indexed Markov chain property, see Theorem \hyperlink{res: peculiarity 2 free}{2c2)}.

How to interpret 
these results? Note that the pinned-at-infinity 
measure 
behaves much differently than a measure would behave
which is obtained by naively pinning the height at a given 
site. HOV-pinned Gibbs measures on trees are much more subtle objects. Note that while site-pinnings will 
also appear in the paper, they do so in an auxiliary way, in the course of proving existence of an HOV.  
Note that such a naively site-pinned measure trivially enjoys 
the tree-indexed Markov chain property away from the pinning site. 
In contrast we prove that this is never the case for the HOV-pinned 
measure. The HOVs in this paper will be constructed as a limit of spherical averages, 
for radii turning to infinity. Viewed in this light, the lack of homogeneity 
of the HOV-pinned measure which we will prove expresses the fact, that on the tree 
even \textit{after the large-sphere limit} the 
choice of the midpoint of the spheres is felt in the distribution. The lack of the tree-indexed Markov property (to be detailed below)
comes from the fact that the fluctuations of the HOV corresponding to the infinite past will necessarily 
be felt on the level of two-spin distributions at adjacent sites. This will become clear in the proofs, see Section \ref{sec: Proof}.

\medskip 
\paragraph{\textbf{Main results 2: Regularity properties of HOVs}} 

Having proved the existence of HOVs we turn to describe their distribution, and also to derive consequences 
for the marginal distribution of the height-variable of the pinned GGM at the origin of the tree. 

The distribution of the HOV 
can be analyzed nicely for the free measure on the level of the characteristic function (and the moment generating function where it exists). 
To do so, we exploit the recursive structure of the tree, from which 
we are able to derive 
Theorem \hyperlink{res: regularity density free}{2b1)} and \hyperlink{res: localization free}{2b2)} (as well as the corresponding Theorem \hyperlink{res: regularity density 2per}{3b1)} and \hyperlink{res: localization 2per}{3b2)}), via an infinite product of the corresponding quantities of $Q/\Vert Q\Vert_1$, see Lemma \ref{pro: singleEdge vs Tree}. 
To round up our analysis, we present two specific cases of interactions, namely the SOS-model and the discrete Gaussian (cf. also e.g., \cite{BaPaRo24} for recent results on the lattice), see Section \ref{sec: applications}. 

The remainder of the paper is organized as follows: Following a short overview of the definitions and notions in Subsection \ref{subsec: def} we present all results regarding pinning at infinity of the free state summarized in Theorem \ref{thm: Free measure}. The corresponding results regarding gradient Gibbs measures of height-period $2$ are then given in an analogous way in Theorem \ref{thm: 2periodic}.
In Section \ref{sec: Proof} we provide the proofs of each of the statements of Theorems \ref{thm: Free measure} and \ref{thm: 2periodic} in the same order as they appear in the respective theorem and structured by subsections whose titles summarize the key idea of proof. Each of these subsections begins with a proof of the respective statement for the easier case of the free state and ends with an extension to the more intricate case of gradient Gibbs measures of height-period $2$.

\section*{Acknowledgement }

We thank an anonymous referee for valuable comments and suggestions, which improved the quality of the paper.

\section{Definitions and main results}
We begin with an overview on the formal framework of our studies. The first part is similar to \cite{HeKu23}, but includes moreover a discussion of the concepts of 
\textit{liftable Gibbs measures} and \textit{height-offset variables} which are particular to the present work, see also \cite{Sh05,LaTo23}.
\subsection{Preliminary definitions}\label{subsec: def}
In the following, we let $\mathbb{T}^d=(V,E)$ denote the Cayley tree of order $d$, i.e., every vertex has exactly $d+1$ nearest neighbors with vertex set $V$, (unoriented) edge set $E$ and root $\rho$. Besides $E$, we also consider the set $\Vec{E}$ consisting of all oriented pairs of vertices connected by an edge. We will write $\{x,y\}$ to denote an unoriented edge and $(x,y)$ to denote an oriented edge.
Furthermore, we will partition $\Vec{E}={}^\rho \Vec{E} \cup \Vec{E}^\rho$, where ${}^\rho \Vec{E}$ denotes those oriented edges which point away from the root $\rho$ and $\Vec{E}^\rho$ is the complementary subset of those edges which point towards $\rho$.

For any $\Lambda \subset V$ we denote by

\begin{equation*}
\partial \Lambda := \{ x \notin \Lambda : d(x,y) = 1 \mbox{ for some } y \in \Lambda\}
\end{equation*} its \textit{outer boundary}.
Furthermore, for any $r \in \N$ we denote by $W_r:=\{v \in V \colon d(\rho,v)=r \}$ the sphere of radius $r$ and by $B_r:=\{v \in V \colon d(\rho,v) \leq r \}$ the closed ball of radius $r$.
Finally, for any $u \in V$ we denote by $S(u):=\{v \in V \colon \{u,v\} \in E \text{ and }d(v,\rho)=d(u,\rho)+1\}$ the set of \textit{children} (direct successors) of $v$.
Within some of the proofs given in Section \ref{sec: Proof} we will also consider the regular $d$-ary tree, where the root has degree $d$ and every other vertex has degree $d+1$.

On the level of edges, we consider the following:
For any $x,y \in V$ we denote by $\Gamma(x,y) \subset \Vec{E}$ the unique directed path from $x$ to $y$.
Furthermore, for any  $\Lambda \subset V$ we define
\begin{align}
E_{\Lambda}&:=\{\{x,y \} \in E \mid x,y \in \Lambda \}, \cr    
\vec{E}_{\Lambda}&:=\{(x,y) \in \vec{E} \mid x,y \in \Lambda \}, \quad \text{and }\cr
\vec{E}_{\Lambda}^\rho&:=\{(x,y) \in \vec{E}^\rho \mid x,y \in \Lambda \}
\end{align} 
as the set of unoriented (oriented, resp.) edges in $\Gamma$ with both end points in $\Lambda$.
Similarly,
by 
\begin{align}
\vec{E}_{\partial \Lambda, \Lambda}&:=\{(x,y) \in E \mid x\in \partial \Lambda, \ y \in \Lambda \}\end{align}
we denote those oriented edges which connect a vertex on the boundary $\partial \Lambda$ to a vertex inside $\Lambda$.

\vspace{5pt}
\paragraph*{\textbf{Height-configurations and gradient configuration}}
Let $\Omega:=\Z^V=\{(\omega_x)_{x \in V} \mid \omega_x \in \Z\}$ denote the set of (integer-valued) \textit{height-configurations} endowed with the product-$\sigma$-algebra $\mathcal{F}=\mathcal{P}(\Z)^{\otimes V}$ generated by the \textit{spin-variables} (projections) $\sigma_y:  \Omega \rightarrow \Z, \ \omega \mapsto \omega_y$, where $y \in V$.

For any arbitrary $\Lambda \subset V$ we denote by $\mathcal{F}_\Lambda$ the $\sigma$-algebra on $\Omega$ generated by the spin-variables for vertices inside $\Lambda$. 
Furthermore, for any arbitrary $\Lambda \subset V$ we denote by $\Omega_\Lambda:=\{(\omega_x)_{x \in \Lambda} \mid \omega_x \in \Z \}$ the set of height configurations inside $\Lambda$. Similarly, $\sigma_\Lambda: \Omega \rightarrow \Omega_\L, \quad \sigma_{\Lambda}((\omega_x)_{x \in V})=(\omega_x)_{x \in \Lambda}$ is the canonical projection to the spins inside $\Lambda$.

Besides the set of height configurations $\Omega$, we consider the set of \textit{gradient configurations}
\begin{equation}\label{eq: gradConf}
\Omega^\nabla:=\{(\zeta_{(u,v)})_{(u,v) \in \vec{E}} \in \Z^{\vec{E}} \mid \zeta_{(u,v)}=-\zeta_{(v,u)}\} \cong \Z^{\Vec{E}^\rho},\end{equation} which describe height increments along directed edges instead of absolute heights. Note that the isomorphy \eqref{eq: gradConf} is clear as the gradient variable on an oriented edge $(u,v)$ is determined 
by the gradient variable on the oriented edge with the same vertices $u,v$ and pointing towards the root. 
The corresponding \textit{gradient-projection} reads $\nabla: \Omega \rightarrow \Omega^\nabla; \ (\omega_x)_{x \in V} \mapsto (\omega_y -\omega_x)_{(x,y) \in \vec{E}}$ with \textit{gradient spin-variables} $\eta_{(x,y)}(\zeta) :=\zeta_{(x,y)}$, which generate the respective gradient sigma-algebra $\mathcal{F}^\nabla:=\mathcal{P}(\Z)^{\otimes \vec{E}^\rho}$ on $\Omega^\nabla$. For any arbitrary $\Lambda \subset V$ we denote by $\mathcal{F}_\Lambda^\nabla$ the product-$\sigma$-algebra generated by the gradient spin variables $(\eta_b)_{b \in \vec{E}^\rho_\Lambda}$.
Given any gradient configuration $\zeta \in \Omega^\nabla$, let $\sigma^{x,s}(\zeta) \in \Omega$ denote the height configuration obtained from $\zeta$ by pinning the height to be $s$ at vertex $x$, i.e., it is given for every vertex $v \in V$ by 
\begin{equation}\label{eq: pinning at x}
\sigma^{x,s}_v(\zeta)=s+\sum_{b \in \Gamma(x,v)}\zeta_b 
\end{equation}
The map $\sigma^{x,s}: \Omega^\nabla \rightarrow \Omega$ thus 
fixes one particular representative in the set of height-configurations with the same gradient configuration, 
and will be frequently used below. In particular, $\Omega^\nabla$ is one-to-one to the set $\Omega/ \Z$ of height-configurations modulo \textit{joint height-shifts} of the form \begin{equation}
\tau_a: \Omega \rightarrow \Omega; 
 \ (\omega_x)_{x \in V} \mapsto (\omega_x-a)_{x \in V}\end{equation} for some $a \in \Z.$ 
\paragraph*{\textbf{Transfer operators and Markovian Gibbsian specifications.}}\hypertarget{par: transfer operator}{}
By a \textit{transfer operator} on $\mathbb{T}^d$ we mean a function $
Q: \Z \rightarrow (0,\infty)$
which is symmetric (i.e., $Q(-i) = Q(i)$ for every $i\in \Z$) and satisfies $Q \in \ell^{\frac{d+1}{2}}(\Z)$ in that $\Vert Q \Vert_{\frac{d+1}{2}}:=\left(\sum_{i \in \Z}Q(i)^\frac{d+1}{2} \right)^\frac{2}{d+1}<\infty$.
The idea behind is that the transfer operator is usually given in terms of a suitable symmetric interaction function $U : \Z \rightarrow [0,+\infty)$ as
\begin{equation}\label{eq: interaction function}
Q(i) = e^{-\beta U(i)},
\end{equation}
where $\beta>0$ is interpreted as the inverse of a temperature. 

A transfer operator $Q$ induces the (local) \textit{Markovian Gibbsian specification}
\[
\gamma = \{\gamma_{\Lambda}: \mathcal{F} \times \Omega \rightarrow [0,1] \}_{\Lambda  \Subset V}
\]
by the assignment
\begin{equation}\label{eq: GibbsSpecification}
\gamma_\Lambda(\sigma_\Lambda= \tilde{\omega}_\Lambda \mid \omega)=\frac{1}{Z_\Lambda(\omega_{\partial \Lambda})}\left(\prod_{\{x,y\} \in E_\Lambda}Q(\tilde{\omega}_x-\tilde{\omega}_y) \right) \, \prod_{(x,y) \in \vec{E}_{\partial \Lambda, \Lambda}}Q(\omega_x-\tilde{\omega}_y), 
\end{equation}
for every $\Lambda\Subset V$, $\tilde\omega \in \Omega$ and $\omega\in \Omega$ acting as a boundary condition. Here, the \textit{partition function} $Z_{\Lambda}$ gives for every $\omega\in \Omega$ the normalization constant $Z_{\Lambda}(\omega_{\partial\Lambda})$ turning $\gamma_{\Lambda}(\cdot \mid \omega)$ into a probability measure on $(\Omega,\mathcal{F})$. The assumption $Q>0$ and the summability assumption $Q\in \ell^{\frac{d+1}{2}}(\Z)$ guarantee that all partition functions are finite and the kernels are therefore well-defined. See Lemma 1 in \cite{HeKu21a}.
As usual, the kernels $\gamma_\Lambda$ are \textit{proper} in the sense of standard Gibbsian theory (e.g., \cite{Ge11}) in that they fix the boundary condition $\omega_{\Lambda^c}$ outside of $\Lambda$.

By construction, the kernels in \eqref{eq: GibbsSpecification} 
are invariant under joint height shifts $\tau_a$ in that \[\gamma_\Lambda\big(\sigma_\Lambda=(\tau_a(\tilde{\omega}))_\Lambda \, \big | \, \tau_a(\omega)\big)=\gamma_\Lambda(\sigma_\Lambda=\tilde{\omega}_\Lambda \mid \omega).\] 
This invariance allows to project each kernel $\gamma_\Lambda$ to a respective \textit{gradient kernel} $\gamma^\nabla$ as follows, where we need to take care of a tree-typical subtlety. 
For this we note that on the tree $(V,E)$ the subgraph $(\Lambda^c, E_{\Lambda^c})$ is disconnected, if the finite subset $\Lambda$ is non-empty. Thus, for arbitrary non-empty $\Lambda \Subset V$ the gradient variables in the complement $\Lambda^c$ (i.e., corresponding to edges in $\vec{E}_{\Lambda^c}$) do not determine the \textit{relative heights at the boundary} $[\omega]_{\partial \Lambda} \in \Z^{\partial \Lambda} \mod \Z$, where $[\tilde{\omega}]_{\partial \Lambda}=[\omega]_{\partial \Lambda}$ if and only if $\tilde{\omega}_{\partial \Lambda}=\left(\tau_a(\omega) \right)_{\partial \Lambda}$ for some $a \in \Z$. Hence, in a gradient kernel on a tree, we need to condition on the gradient field outside $\Lambda$, $(\eta_b)_{b \in \vec{E}_{\Lambda^c}}$ \textbf{and} additionally on sums of gradients for those edges in $\vec{E}_{\Lambda \cup \partial \Lambda}$ that lie on a path that connects any two vertices $x,y$ at the boundary $\partial \Lambda$, as those sums express the relative height between $x$ and $y$ in terms of gradients. More formally, we define the equivalence relation $[\cdot]_{\partial \Lambda}$ now on the level of gradient configurations as
\begin{equation*}
[\tilde{\zeta}]_{\partial \Lambda}=[\zeta]_{\partial \Lambda} \quad \text{ if and only if } \sum_{b \in \Gamma(x,y)}\tilde{\zeta}_b = \sum_{b \in \Gamma(x,y)}\zeta_b   
\end{equation*}
for all $x,y \in \partial \Lambda$.
Then note that $[\tilde{\omega}]_{\partial \Lambda} = [\omega]_{\partial \Lambda}$ if and only if $[\nabla\tilde{\omega}]_{\partial \Lambda} = [\nabla \omega]_{\partial \Lambda}$. In words, the equivalence class $[\zeta]_{\partial \Lambda}$ of gradient configurations defines a unique configuration of relative heights at the boundary.
In what follows we call the respective $\sigma$-algebra \begin{equation} \mathcal{T}_\Lambda^\nabla:=\sigma\left((\eta_{(x,y)})_{(x,y) \in \vec{E}_{\Lambda^c}}, \ [\eta]_{\partial \Lambda} \right) \end{equation}
on $\Omega^\nabla$ the \textit{outer gradient $\sigma$-algebra for $\Lambda$.}
Note that $\mathcal{T}_\Lambda^\nabla$ is strictly bigger than $\mathcal{F}^\nabla_{\Lambda^c}$. This is a particularity of the tree and does not happen on higher-dimensional lattices.\\
Hence, the gradient kernels $\gamma^\nabla_\Lambda$ are 
\begin{equation}
\gamma^\nabla_\Lambda(\eta_{\Lambda \cup \partial \Lambda}= \zeta_{\Lambda \cup \partial \Lambda} \mid \zeta):=\gamma_\Lambda(\sigma_\Lambda =\omega_\Lambda \mid \omega)    
\end{equation}
for any $\omega \in \Omega$ such that $(\nabla \omega)_{\Lambda^c}=\zeta_{\Lambda^c}$ and $[\nabla \omega]_{\partial \Lambda}=[\zeta]_{\partial \Lambda}$.
More generally, for every bounded continuous function $F$ on $\Omega^\nabla$ we have
\begin{equation}\label{grad}
	\int_{\Omega^\nabla} F(\rho) \gamma^\nabla_\Lambda(\text{d}\rho \mid \zeta) = \int_\Omega F(\nabla \varphi) \gamma_\Lambda(\text{d}\varphi\mid\omega)
	\end{equation}
     where $\omega \in \Omega$ is any height-configuration with $(\nabla \omega)_{\Lambda^c} = \zeta_{\Lambda^c}$ and $[\nabla \omega]_{\partial \Lambda}=[\zeta]_{\partial \Lambda}$.
This defines a family of probability kernels $(\gamma^\nabla_\Lambda)_{\Lambda \Subset V}$ from $(\Omega^\nabla,\mathcal{T}^\nabla_\Lambda)$ to $(\Omega^\nabla,\mathcal{F}^\nabla)$, which we call the \textit{gradient Gibbs specification} associated with $Q$. For further illustration, in Section \ref{App: measurability} we explicitly consider \eqref{grad} in the special case of a one-element volume.
\vspace{5pt}
\paragraph*{\textbf{Gibbs measures and gradient Gibbs measures.}}
Following Dobrushin \cite{Do68} and Landford and Ruelle \cite{LaRu69}, a probability measure $\mu$ on $(\Omega,\mathcal{F})$ is called a \textit{Gibbs measure} for the specification $\gamma$ iff the \textit{DLR-equation} holds
\begin{equation}
\int_\Omega F(\omega)\mu(d\omega)=\int_\Omega \mu(d \tilde{\omega})\int_\Omega F(\tilde{\omega})\gamma_\Lambda(d\tilde{\omega} \mid \omega)
\end{equation}
for every bounded continuous function $F$ on $\Omega$ and every $\Lambda \Subset V$.
Similarly, a gradient measure $\nu$ on $(\Omega^\nabla,\mathcal{F}^\nabla)$
is called a \textit{gradient Gibbs measure} for the gradient specification $\gamma^\nabla$ if and only if
\begin{equation}
\int_{\Omega^\nabla} F^\nabla(\zeta)\nu(d\zeta)=\int_{\Omega^\nabla} \nu(d \zeta)\int_{\Omega^\nabla} F(\tilde{\zeta})\gamma_\Lambda(d\tilde{\zeta} \mid \zeta)
\end{equation}
for every bounded continuous function $F^\nabla$ on $\Omega^\nabla$ and every $\Lambda \Subset V$.

\paragraph*{\textbf{Tree-indexed Markov chains.}}
Following \cite[Chapter 12]{Ge11}, choosing any oriented edge $(u,v) \in \vec{E}$ of the tree induces a splitting of the sets of vertices $V$ into \textit{the past}
\begin{equation}\label{eq: the past}
[-\infty,uv):=\{w \in V \mid \text{d}(w,u)<\text{d}(w,v)\}
\end{equation}
and \textit{the future} $[uv,+\infty):=V \setminus [-\infty,uv)$. See also Figure \ref{Fig: SpliitingIntoSubtrees} below for an illustration.
Then, a probability measure $\mu$ on $(\Omega,\mathcal{F})$ is by definition a \textit{tree-indexed Markov chain} iff for every $(u,v) \in \vec{E}$
\begin{equation}
\mu(\sigma_v=\cdot \mid \mathcal{F}_{[\infty,uv)})=\mu(\sigma_v=\cdot \mid \mathcal{F}_{\{ u \}}) \quad \mu-\text{a.s.}    
\end{equation}
\vspace{5pt}
\paragraph*{\textbf{Liftable gradient Gibbs measures}}
While every Gibbs measure can be projected to a gradient Gibbs measure for the respect gradient specification via the map $\nabla$, which sends a spin configuration to its gradients, the converse is not generally true. This distinction corresponds to whether a gradient Gibbs measure is \textit{liftable} (or \textit{smooth} in the sense of Sheffield\cite{Sh05}) in that it can be lifted to a proper Gibbs measure, or not.
\begin{defn}\label{def: HeightOffset}
A function $H: \Omega \rightarrow \R \cup \{\pm \infty\}$ is called a \underline{height-offset variable} for a gradient Gibbs measure $\nu$ iff it fulfills the following properties:
\begin{enumerate}[a)]
\item For all $a \in \Z$ we have $H(\tau_{-a}(\omega))= H(\omega)+a$ for all $\omega \in \Omega$;
\item $H$ is measurable with respect to the tail-$\sigma$-algebra $\mathcal{T}$ on $\Omega$;
\item $\nu( \vert H \circ\sigma^{x,s}  \vert <\infty)=1$ for some site $x\in V$ and some 
integer height $s\in \Z$.
\end{enumerate}
\end{defn}
Here, $\sigma^{x,s}$ is defined as in \eqref{eq: pinning at x}.
For c) note that by a) the gradient event of finite HOV 
$\{\vert H \circ\sigma^{x,s}  \vert <\infty\}\in \mathcal{F}^\nabla$ 
is independent of the choice of the auxiliary pinning site $x$ and the pinning height 
$s$. Thus c) can be equivalently written by fixing $x$ to be the root $\rho$, and $s=0$ to become 
the requirement 
$\nu( \vert H \circ\sigma^{\rho,0}  \vert <\infty)=1$. 
For a given height-offset variable $H$ let 
$\Omega^\nabla_H:= \Omega^\nabla \cap 
\vert H \circ\sigma^{\rho,0}  \vert <\infty)
\}$ be the associated set of \textit{good gradient configurations} where $H$ is well-defined. 
Then, for fixed HOV $H$, and for any fixed auxiliary choices of site $x \in V$ and height 
$s \in \Z$ define  the pinning-at-infinity map  carrying us from gradients 
to height configurations as follows. 
\begin{equation}\label{eq: GGMtoGM}
G_H: \Omega_H^\nabla \rightarrow \Omega; \ G_H(\zeta):=\sigma^{x,s}(\zeta)-\lfloor H(\sigma^{x,s}(\zeta)) \rfloor,
\end{equation}
where $\lfloor \alpha \rfloor$ denotes the largest integer less or equal to $\alpha \in \R$.

Let us comment on this definition.  
Note that the HOV $H$ appearing in the definition of $G_H$ takes as arguments height-configurations, and not gradient configurations. 
 In the definition of $G_H$ 
 the height configurations which are obtained from a gradient configuration $\zeta$ by means of 
$\sigma^{x,s}(\zeta)$ appear in two places.  
This is just a necessary auxiliary  pinning of the height to be $s$ in the vertex $x$ which at first glance 
could lead to $x,s$-dependent results.  
Note however that by the defining properties of a height-offset variable,  
the definition of the function $G_H$ is independent of the arbitrary choices of $x$ and $s$,
as the difference $\sigma^{y,t}(\zeta)-\sigma^{x,s}(\zeta)=(t-s)-\sum_{b \in \Gamma(x,y)}\zeta_b$ drops out of the defining equation \eqref{eq: GGMtoGM}.
In particular, \eqref{eq: GGMtoGM} implies that the image measure 
\begin{equation}\label{eq: pinned measure muh}
\mu^H:=\nu \circ G_H^{-1} \in  \mathcal{M}_1(\Omega,\mathcal{F})    
\end{equation} has single-site marginals given by
\begin{equation}\label{eq: SingSiteMarginals}
\mu^H(\sigma_x=t)=\nu(\, \lfloor H(\sigma^{x,0}) \rfloor=- t)=\nu(\sum_{b \in \Gamma(\rho,x)}\eta_b-\, \lfloor H(\sigma^{\rho,0}) \rfloor=t), \quad t \in \Z.    
\end{equation}
$\mu^H$ is the measure on the height configurations which is obtained from the gradient measure $\nu$ by means of the height-offset variable $H$, which provides a pinning at infinity 
which is suitable to preserve the Gibbs property, as we will discuss in a moment.
We see that $\mu^H$ is obtained from the gradient Gibbs measure by adding fluctuating localization centers $\lfloor H(\sigma^{\rho,0}) \rfloor$.

The relevance of the defining properties of a height-offset variable lies in the fact that its properties imply the Gibbs property of the measure $\mu^H$ in the space of height configurations. More precisely we have the following theorem.
\begin{thm}[Gibbs property of measures pinned at infinity]\label{thm: DLR}
Let $\nu$ be a gradient Gibbs measure for $\gamma^\nabla$ which possesses a height-offset variable $H$. 

Then the image measure $\mu^H=\nu \circ G_H^{-1} \in  \mathcal{M}_1(\Omega,\mathcal{F})$ is a Gibbs measure for $\gamma$.
\end{thm}
The Theorem goes back to Sheffield \cite{Sh05} and Lammers and Toninelli \cite{LaTo23}, but for convenience of the reader we include a detailed proof in the Appendix. 
Here we pay particular attention to the tree-specific 
measurability issues of the gradient kernels, which stem from the disconnectedness 
of the complements of finite volume, as explained above.  

\begin{rk}
Theorem \ref{thm: DLR} above does not impose any assumptions on regularity of the tree $(V,E)$ or on symmetries of the gradient Gibbs measure $\nu$.
\end{rk}
We will construct height-offset variables as the limit of spherical averages of height configurations. To verify that this limit exists as a real-valued random variable on $\Omega$ $\nu$-almost surely, we will also consider spherical averages of gradient fields with a prescribed absolute height at the root. 
\begin{defn}[Spherical averages of height- and pinned gradient configurations]\label{def: SphericalAverage}\hfill \\
\begin{enumerate}[a)] 
    \item We denote by $H_r: \Omega \rightarrow \R$ 
\begin{equation}
H_r(\omega):=\frac{1}{\vert W_r \vert}\sum_{v \in W_r}\omega_v 
\end{equation}
the average value of the heights at the $r$-sphere, $r \in \N$.
\item Furthermore, by $H^{\nabla}_r: \Omega^\nabla \rightarrow \R$
\begin{equation}
H^{\rho,0}_r(\zeta):=\frac{1}{\vert W_r \vert}\sum_{v \in W_r}\sigma_v^{\rho,0}  (\zeta)  
\end{equation}
we denote the average value of the heights of the gradient field pinned at height zero at the root $\rho$, $r \in \N$.
\end{enumerate}
\end{defn}
\begin{rk}
We have $H_r(\omega)=H_r^{\rho,0}(\nabla \omega)+\omega_\rho$ for every $r \in \N$ and $\omega \in \Omega$.
\end{rk}
\begin{assumption}[Second moment condition on transfer operator $Q$]\label{ass: SquareIntegrability}
In the following Theorems \ref{thm: Free measure} and \ref{thm: 2periodic} we assume that
\begin{equation}
\sum_{j \in \mathbb{Z}}j^2Q(j) < \infty,    
\end{equation}
which means that the probability measure on $\Z$ described by the probability mass function $\frac{Q(\cdot)}{\Vert Q \Vert_1}$ has finite second moments.
\end{assumption}
\subsection{Result: Pinning at infinity for free state under second moment assumption on transfer operator}
At first, we show that the gradient Gibbs measure which arises for free boundary conditions can be pinned at infinity.
\begin{defn}[Free state of i.i.d-increments]
Given a positive symmetric transfer operator $Q \in \ell^\frac{d+1}{2}(\mathbb{Z})$, we call the probability measure $\nu$ on $(\Omega^\nabla, \mathcal{F}^\nabla)$ with marginals     
\begin{equation}
		\nu(\eta_{\Lambda \cup \partial \Lambda}=\zeta_{\Lambda \cup \partial \Lambda})=\frac{1}{Z_\Lambda}\prod_{(x,y) \in {}^\rho\Vec{E}_{\Lambda \cup \partial \Lambda}}Q(\zeta_{(x,y)}).
	\end{equation}
 \textbf{the free state (of i.i.d.-increments)}. 
\end{defn}

With that notation in mind, the result reads
\begin{thm}[Pinning at infinity of the free gradient state: existence, regularity of density, loss of translation invariance and of Markov property]\label{thm: Free measure}\mbox{}\\
Let $Q: \Z \rightarrow (0,\infty)$ be any positive symmetric transfer operator of gradient type such that $\sum_{j \in \N}j^2 Q(j) < \infty$.
Then the following holds true for the corresponding free state $\nu$ on the Cayley tree of order $d \geq 2$.
 \begin{itemize}
	\item[a)]\textbf{Existence.}\hypertarget{res: existence free}{} Let $(H_r)_{r \in \mathbb{N}}$ be as in Definition \ref{def: SphericalAverage}. Then $H:=\lim_{r \rightarrow \infty}H_r$ is a height-offset variable for $\nu$ in the sense of Definition \ref{def: HeightOffset}, where the limit holds almost surely and in $\mathcal{L}^2$ with respect to the same measure $\nu  \circ \left(\sigma^{\rho,0} \right)^{-1}$. In particular, the measure $\mu^H$ associated with $H$ and $\nu$ via Theorem \ref{thm: DLR} is a proper Gibbs measure on the space of height configurations.
 \item[b1)]\textbf{Regularity of density of HOV.}\hypertarget{res: regularity density free}{} The corresponding distribution of $H^{\rho,0}$ obtained from Definition \ref{def: SphericalAverage} under $\nu$ has an infinitely often differentiable Lebesgue-density $f_{H^{\rho,0}}$.\\
 Moreover, $f_{H^{\rho,0}}$ is strictly positive on $\R$.
 \item[b2)]\textbf{Localization of single-site marginal of pinned measure.}\hypertarget{res: localization free}{} Assume furthermore on the interaction $Q$ that there is some $R>0$ such that the moment generating function for a single gradient variable, $t \mapsto \sum_{j \in \Z}\exp(tj)\frac{Q(j)}{\Vert Q \Vert_1}$ is finite for $t \in (-R,+R)$. \\
 Then, for every $A < (d+1)R$ there exists a $B(A) \in (0,\infty)$ such that the single-site marginal at the root satisfies the following exponential concentration:
 \[\mu^H(\vert \sigma_\rho \vert \geq s) \leq B(A)\exp(-As), \quad \text{for every } s \in \mathbb{N}.\]
  \item[c1)]\textbf{Peculiarity 1.}\hypertarget{res: peculiarity 1 free}{} The lifted measure Gibbs $\mu^H$ is necessarily not invariant under the translations of the tree.
 \item[c2)]\textbf{Peculiarity 2.}\hypertarget{res: peculiarity 2 free}{} The measure $\mu^H$ is not a tree-indexed Markov chain and in particular not extreme in the set of all Gibbs measures for the specification associated with $Q$ via \eqref{eq: GibbsSpecification}.
 \end{itemize}
\end{thm}
\subsection{Result: Pinning at infinity for 2-height periodic GGMs under second moment assumption on transfer operator}\label{subsec: results 2-height periodic} 
For general integer-valued gradient models on regular trees, described via a transfer operator $Q$, 
\cite{HKLR19,HeKu21a,HeKu23,AbHeKuMa24}, provided period-$q$ height-periodic GGMs for $Q$ via a two-layer hidden Markov model 
construction with an internal layer $\Z_q$. The hidden variables on the internal layer interact via a suitably chosen $q$-dependent fuzzy transfer operator $Q^q$. The  full measure on integer-valued increments 
is obtained from here via application of transition kernels acting on the edges. 
The classes of these GGMs associated with different coprime values of $q$ were then shown to be disjoint. Furthermore, for any summable homogeneous $Q$ there always exist non-translation invariant gradient Gibbs measures corresponding to an internal layer with sufficiently large (depending on $Q$) $q$ (see \cite{HeKu23}).


We shortly present the construction for the special case of period-2-height periodic GGMs. 

These are GGMs for the same interaction 
$Q$ as the free states we discussed before. 
Here the internal hidden layer variables 
have a distribution in infinite volume 
which turns out 
to be given by an Ising model, at 
a suitably chosen inverse temperature, as 
we will describe. 
The appearance of such an Ising model
becomes clear when the usual transformation between 
lattice gas variables and Ising variables 
is made for the internal variables.

Afterwards we  state the new result, Theorem \ref{thm: 2periodic}, which analogously to Theorem \ref{thm: Free measure} above for the free state describes the existence and further properties of a suitable HOV for this class of GGMs.

Let $\mathbb{Z}_2=\mathbb{Z}/2\mathbb{Z}=\{\bar{0},\bar{1}\}$. For any positive symmetric $Q \in \ell^1(\Z)$ define the associated \textit{fuzzy transfer operator} $\bar{Q}: \Z_2 \rightarrow (0,\infty)$ by setting 
\begin{equation}\label{eq: barQ def}
\bar{Q}(\bar{0}):=\sum_{l \in \Z}Q(2l) \text{ and } \bar{Q}(\bar{1}):=\sum_{l \in \Z}Q(1+2l).
\end{equation}    
The fuzzy transfer operator $\bar{Q}$ describes an Ising model in the following sense:
if we identify $\bar{0}$ with $-1$ and $\bar{1}$ with $+1$ then the formal Ising-Hamiltonian  
\[\mathcal{H}(\tilde{\omega})=-J\sum_{x \sim y}\tilde{\omega}_x\tilde{\omega}_y, \quad \tilde{\omega} \in \{\pm 1\}^V\]
is built from $\bar{Q}$ via \[\exp(2J)=\frac{\bar{Q}(\bar{0})}{\bar{Q}(\bar{1})}.\]
In the following we will write $\bar{a}:=a+2\Z$ to denote the projection of an integer $a$ to $\Z_2$.
Fix any Gibbs measure $\bar{\mu}$ on $\Z_2^V$ with spin-variables $(\bar{\sigma}_v)_{v \in V}$ for the Ising model described by $\bar{Q}$. Then we consider the following two-step procedure, where we first sample a configuration of $\mathbb{Z}_2$-valued spins (\textit{fuzzy chain}) which plays the role of a random environment.
Given a realization of such an environment we then sample edge-wise independently integer-valued increments conditionally on whether the fuzzy chain changes or preserves its value along the respective edge according to the kernel 
\begin{equation}
\rho^Q(s \mid \bar{a}):=\boldsymbol{1}_{a+2\Z}(s)\frac{Q(s)}{\bar{Q}(\bar{a})}, \quad s \in \mathbb{Z}, \ \bar{a} \in \{\bar{0},\bar{1} \}.
\end{equation}
\begin{defn}[2-height periodic GGM]\label{def: 2perGGM}
Let $\bar{\mu}$ be a Gibbs measure for the Ising model defined by $\bar{Q}$. Then we call the annealed measure $\nu^{\bar{\mu}}$ on $(\Omega^\nabla, \mathcal{F}^\nabla)$ with marginals
\begin{equation}\label{eq: def2perGGM}
\begin{split}
\nu^{\bar{\mu}}(\eta_{\Lambda \cup \partial \Lambda}=\zeta_{\Lambda \cup \partial \Lambda})=\frac{1}{Z_{\Lambda \cup \partial \Lambda}}\sum_{\bar{\omega}_{\Lambda \cup \partial \Lambda} \in \mathbb{Z}_2^{\Lambda \cup \partial \Lambda}}\bar{\mu}(\bar{\sigma}_{\Lambda \cup \partial \Lambda}=\bar{\omega}_{\Lambda \cup \partial \Lambda})
\prod_{(x,y) \in {}^\rho\Vec{E}_{\Lambda \cup \partial \Lambda}}\rho^Q(\zeta_{(x,y)} \mid \bar{\omega}_y-\bar{\omega}_x),
\end{split}
\end{equation}
where $\Lambda \Subset V$, the \textbf{gradient Gibbs measure of height period $2$ associated with the hidden Ising measure (or \textit{fuzzy chain}) $\bar{\mu}$}.

\end{defn}

The gradient measure $\nu^{\bar{\mu}}$ indeed satisfies the DLR-equation, for a proof see  \cite[Theorem 2]{HeKu23}, where we note that the construction does not impose any further assumption such as spatial homogeneity on the hidden Ising measure $\bar{\mu}$ and, even more, generalizes from the hidden Ising model to more general hidden clock models on $\mathbb{Z}_q$, $q \geq 2$.

Note that there is an appealing 
alternative view to the two-periodic GGM 
as coming from a mixture of branching random walks.  
Although we won't use this description further for the proofs 
in our current paper, let us give a brief informal description. 
In the branching random walk view
each infinite path on the tree takes the role of a time $\Z$, 
the random walks move in $\Z$, with transition probabilities 
which are two-periodic in the height direction. 
The process on the tree is then obtained by glueing 
all these paths together consistently. 
For more details see \cite{KS17}.

Now the theorem reads:
\begin{thm}[Pinning at infinity for gradient Gibbs measures of height period $2$: existence, regularity of density, loss of translation invariance and loss of extremality]\label{thm: 2periodic}\mbox{}\\
Let $Q: \Z \rightarrow (0,\infty)$ be a positive symmetric transfer operator of gradient type such that $\sum_{j \in \N}j^2 Q(j) < \infty$.
Let $\bar{\mu}$ be any (not necessarily spatially homogeneous) tree-indexed Markov chain Gibbs measure for the Ising model with interaction described by $\bar{Q}$ (see Eq. \eqref{eq: barQ def}). Then the following holds true for the gradient Gibbs measure $\nu^{\bar{\mu}}$ of height-period $2$ associated to $\bar{\mu}$ on the $d$-regular tree with $d \geq 2$ via Definition \ref{def: 2perGGM}.
 \begin{itemize}
	\item[a)]\textbf{Existence.}\hypertarget{res: existence 2per}{}  Let $(H_r)_{r \in \mathbb{N}}$ be as in Definition \ref{def: SphericalAverage}. Then $H:=\lim_{r \rightarrow \infty}H_r$ is a HOV for the measure $\nu^{\bar{\mu}}$ where the limit holds almost surely and in $\mathcal{L}^2$ with respect to the same measure $\nu^{\bar{\mu}} \circ \left(\sigma^{\rho,0} \right)^{-1}$. In particular, the measure $\mu^H$ associated with $H$ and $\nu$ via Theorem \ref{thm: DLR} is a proper Gibbs measure on the space of height configurations. 
 \item[b1)]\textbf{Regularity of density of HOV.}\hypertarget{res: regularity density 2per}{} The corresponding distribution of $H^{\rho,0}$ obtained from Definition \ref{def: SphericalAverage}, under $\nu^{\bar{\mu}}$, has an infinitely often differentiable Lebesgue-density $f_{H^{\rho,0}}$. \\
 Moreover, $f_{H^{\rho,0}}$ is strictly positive on $\R$.
 \item[b2)]\textbf{Localization of single-site marginal of pinned measure.}\hypertarget{res: localization 2per}{} 
 Assume furthermore on the interaction $Q$ that there is some $R>0$ such that the function $t \mapsto \sum_{j \in \Z}\exp(tj)\frac{Q(j)}{\Vert Q \Vert_1}$ is finite for $t \in (-R,+R)$.
 Then, for every $A < (d+1)R$ there exists a $B(A) \in (0,\infty)$ such that the single site marginal at the root satisfies the following exponential decay:
 \[\mu^H(\vert \sigma_\rho \vert \geq s) \leq B(A)\exp(-As), \quad \text{for every } s \in \mathbb{N}.\]
  \item[c1)]\textbf{Peculiarity 1.}\hypertarget{res: peculiarity 1 2per}{} The lifted Gibbs measure $\mu^H$ is necessarily not invariant under the translations of the tree.
 \item[c2)]\textbf{Peculiarity 2.}\hypertarget{res: peculiarity 2 2per}{} The measure $\mu^H$
is not extreme in the set of all Gibbs measures for $Q$.
 \end{itemize}
\end{thm}
\section{Proofs}\label{sec: Proof}
We will prove each of the Statements (a), (b1), (b2), (c1) and (c2) of the Theorems \ref{thm: Free measure} and \ref{thm: 2periodic} in a separate subsection.
We begin with the first part of Statement (a) on almost sure convergence of the family $(H_r)_{r \in \mathbb{N}}$, which follows from identifying the process as a suitable $\mathcal{L}^2$-Martingale for a suitable filtration. Afterwards, we study the moment generating functions associated with $(H_r)_{r \in \mathbb{N}}$ to obtain Statement (b2) on localization of the single-site marginals of the measure $\mu^H$. Subsequently, we give the proof of the first part of Statement (b1) on absolute continuity of the distribution of the HOV $H$, which follows from considering the characteristic functions associated with $(H_r)_{r \in \mathbb{N}}$, and is thus merely a corollary of the previous calculations with moment generating functions (two-sided Laplace transforms).
We first present the more illustrative proof of Statement (c1) on non-invariance of the pinned measure under the translations of the tree. This is based on splitting the tree into the two sub-trees of the past and of the future of an edge and analyzing a reflection that changes the roles of infinite past and future.
This is followed by the proof of the second part of Statement (b1) which relies on another decomposition of the tree.
Finally, we conclude with a proof of Statement (c2) on violation of the Markov property and non-extremality. 
  
We present the respective steps for the free state and the 2-height periodic GGMs as much in parallel as possible to avoid redundancies in argumentation and to keep presentation short and concise.  
\subsection{Averages over spheres form a suitable $\mathcal{L}^2$-martingale}
To show that $\limsup_{r \rightarrow \infty} H_r$ is $\nu$-almost surely finite, in what follows, we consider the function $H_r^{\rho,0}$ (cf. Definition \ref{def: SphericalAverage}) which maps any gradient configuration $\zeta$ to the average height at the $r$-sphere, where we pin the height $0$ at the root $\rho$. 
\begin{pro}[Martingale construction for the free state]\label{pro: MartingaleFree}
Let $Q$ be a positive symmetric transfer operator with $\sum_{j=0}^\infty j^2 Q(j)<\infty$. Then for the free state $\nu^{free}$, the family $(H_r^{\rho,0})_{r \in \mathbb{N}}$ is an $\mathcal{L}^2(\nu^{free})$-martingale with respect to the filtration $(\mathcal{F}^\nabla_{B_r})_{r \in \mathbb{N}}$. Hence, $(H_r^{\rho,0})_{r \in \mathbb{N}}$ converges $\nu^{free}$-a.s. and in $\mathcal{L}^2(\nu^{free})$ to a limit $H^{\rho,0}$, which is $\nu^{free}$-a.s. finite.
\end{pro}
\begin{proof}[Proof of Proposition \ref{pro: MartingaleFree}]
The family $(H_r^{\rho,0})_{r \in \mathbb{N}}$ is by construction adapted to the filtration $\mathcal{F}^\nabla_{B_r}$. To prove the martingale property, let $r \in \N$. Then linearity of the conditional expectation and the Markov property give $\nu^{free}$-a.s.
 \begin{equation} 
 \begin{split}\label{eq: Martingale property}
 \mathbb{E}[H^{\rho,0}_{r+1} \mid \mathcal{F}^\nabla_{B_r}]&=\frac{1}{\vert W_r \vert } \sum_{u \in W_r} \frac{1}{d}\sum_{v \in S(u)} \mathbb{E}[\sigma_v^{\rho,0} \mid   \mathcal{F}^\nabla_{B_r}] 
 = \frac{1}{\vert W_r \vert} \sum_{u \in W_r} \frac{1}{d}\sum_{v \in S(u)} \mathbb{E}[\sigma_v^{\rho,0} \mid   \sigma_u].
 \end{split}
 \end{equation}
 From the fact that $\nu(\sigma_x=i \mid \sigma_y=j)=\frac{Q(i-j)}{\Vert Q \Vert_1}$ and symmetry of $Q$ it follows that for every $u \in W_r$ and $v \in S(u)$
\begin{equation} \label{eq: condExp}
 \mathbb{E}[\sigma_v^{\rho,0} \mid   \sigma_{u}^{\rho,0}]=\sigma_u^{\rho,0} \quad \nu^{free}{\text-a.s.}
 \end{equation}
 Inserting this into \eqref{eq: Martingale property} shows the martingale property.
 
 In the second step, we show $\mathcal{L}^2(\nu^{free})$-boundedness of the martingale $(H_r^{\rho,0})_{r \in \N}$ to deduce from the martingale convergence theorem \cite[Corollary 11.11]{Kl20} that $(H_r^{\rho,0})_{r \in \N}$ converges $\nu^{free}$-a.s. to a real-valued limiting random variable.

 We have 
 \begin{equation}
 \begin{split}
 \mathbb{E}[(H_r^{\rho,0})^2]=\frac{1}{\vert W_r \vert^2}\sum_{u,v \in W_r}\mathbb{E}[\sigma_u^{\rho,0} \sigma_v^{\rho,0}]&=\frac{1}{\vert W_r \vert}\sum_{v_2 \in W_r}\mathbb{E}[\sigma_{v_1}^{\rho,0}\sigma_{v_2}^{\rho,0}],
 \end{split}
 \end{equation}
 where $v_1 \in W_r$ is any fixed vertex.
 For any $v_2 \in W_r$ the sets of vertices on the shortest paths from the root $\rho$ to $v_1$ and from $\rho$ to $v_2$ have a nonempty intersection (at least the root is contained in the intersection). 
 Let 
 \begin{equation}\label{eq: SplittingPointFree}
 u(v_2):=u_{v_1}(v_2)
 \end{equation}
 denote the vertex in the intersection which maximizes the distance to root, i.e., the point at which the two paths split and define the function $\mathbbm{u}: V \rightarrow \N_0$ by setting $\mathbbm{u}(v):=d(\rho,u_{v_1}(v))$.
Now, by conditional independence (from the Markov property) and \eqref{eq: condExp} we have for all $v_2 \in W_r$
\begin{equation}
\begin{split}
\mathbb{E}[\sigma_{v_1}^{\rho,0}\sigma_{v_2}^{\rho,0}]&=\mathbb{E}[\,\mathbb{E}[\sigma_{v_1}^{\rho,0}\sigma_{v_2}^{\rho,0} \mid \sigma_{u(v_2)}^{\rho,0}]\,]=\mathbb{E}[\,\mathbb{E}[\sigma_{v_1}^{\rho,0} \mid \sigma_{u(v_2)}^{\rho,0}]\cdot \mathbb{E}[\sigma_{v_2}^{\rho,0} \mid \sigma_{u(v_2)}^{\rho,0}]\,] =\mathbb{E}\big[(\sigma_{u(v_2)}^{\rho,0} )^2 \big].
\end{split}
\end{equation}
To calculate the second moment of the spin at site $u(v_2)$ let $Y_1, \ldots, Y_{\vert u(v_2) \vert }$ denote the increments along the edges of the path from $\rho$ to $u(v_2)$. As these are independent under the free measure $\nu^{free}$ we arrive at 
\begin{equation}
\mathbb{E}\bigg[\left(\sigma^{\rho,0}_{u(v_2)}\right)^2\bigg]=\mathbbm{u}(v_2) \, \mathbb{E}[Y_1^2].
\end{equation} 
Hence, 
\begin{equation}\label{eq: secMoment}
 \begin{split}
 \mathbb{E}[(H_r^{\rho,0})^2]&=\frac{1}{\vert W_r \vert}\mathbb{E}[Y_1^2]\sum_{v_2 \in W_r} \mathbbm{u}(v_2)=\mathbb{E}[Y_1^2]\sum_{n=1}^r n \, \frac{\vert \{ v_2 \in W_r \mid \mathbbm{u}(v_2)=n\} \vert}{\vert W_r \vert}.
 \end{split}
 \end{equation}
To calculate the sizes of the level sets of the function $\mathbbm{u}$ we proceed with the following iteration from the outside of the tree to the inside (see Figure \ref{fig: Recursion}): 
 $v_1$ is the only vertex in $W_r$ for which the function $\mathbbm{u}$ reaches its maximum $r$. Going over to the parent $\tilde{v}$ of $v_1$, we see that there are $d-1$ vertices (which are the children of $\tilde{v}$ except $v_1$) for which $\mathbbm{u}$ takes the value $r-1$. Iterating this argument (if $r \geq 3$) then means going over to the parent $\vardbtilde{v}$ of $\tilde{v}$ and counting the total offspring in $W_r$ of all of its children except $\tilde{v}$, which amounts to $(d-1)d$ vertices at which $\mathbbm{u}$ takes the value $r-2$.  Any further iteration (as long as we do not arrive at the root $\rho$) corresponds to a multiplication by $d$.
 Consequently 
 \begin{equation}\label{eq: SplitPoint}
  \vert \{ v_2 \in W_s \mid \mathbbm{u}(v_2)=n\} \vert=\begin{cases}
  	(d-1)d^{r-n-1}, \quad &\text{if } 1 \leq n \leq r-1;\\
  	  	1 , \quad &\text{if }  n=r. \\
  \end{cases}
\end{equation}
\begin{figure}
    \centering
    \includegraphics[width=0.35\linewidth]{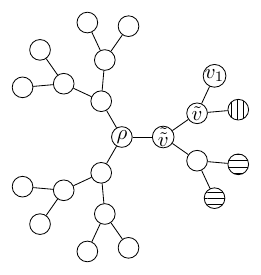}
    \caption{An illustration for the recursion leading to Eq. \eqref{eq: SplitPoint}. The graphics shows the ball $B_r$ of radius $r=3$ on the Cayley tree of order $d=2$ with some distinguished vertex $v_1 \in W_r$. The vertex $v_1$ is the unique vertex for which the function $\mathbbm{u}$ takes its maximal value $r$. The $d-1=1$ vertices filled with vertical lines are the vertices for which  $\mathbbm{u}$ takes the value $r-1$. The $d(d-1)$ vertices filled with horizontal lines are the vertices for which $\mathbbm{u}$ takes the value $r-2$. The recursion stops here, as $r-3=0$.}
    \label{fig: Recursion}
\end{figure}
In combination with $\vert W_r \vert=(d+1)d^{r-1}$ \eqref{eq: secMoment} hence becomes
\begin{equation}\label{eq: secMoment2}
 \mathbb{E}[(H_r^{\rho,0})^2]=\frac{d-1}{d+1}\mathbb{E}[Y_1^2]\sum_{n=1}^r nd^{-n}<\frac{d-1}{d+1}\mathbb{E}[Y_1^2]\sum_{n=1}^{\infty} nd^{-n}  =2\frac{d}{d^2-1}\frac{\sum_{j =1}^\infty j^2Q(j)}{\Vert Q \Vert_1},   
\end{equation}
where the r.h.s of \eqref{eq: secMoment2} is independent of $r$ and, by the assumption on $Q$, finite.
In particular, the martingale 
 $(H_r^{\rho,0})_{r \in \N}$ is $\mathcal{L}^2(\nu^{free})$-bounded and hence converges both a.s. and in $\mathcal{L}^2(\nu^{free})$. This concludes the proof of Proposition \ref{pro: MartingaleFree}.
  \end{proof}

We now come to the treatment of the height-periodic GGMs which is a little 
more complex.
While we are mostly interested in the annealed measure $\nu^{\bar \mu}$ \eqref{eq: def2perGGM}, 
it will be convenient for our further analysis to consider also 
the following \textit{joint infinite-volume 
measure} $\bar\nu^{\bar \mu}$ on fuzzy spins and integer-valued 
increments (that carries an additional bar in the notation) which by definition has finite-volume marginals
\begin{equation}\label{eq: 2perGGMProd}
\begin{split}
&\bar{\nu}^{\bar{\mu}}(
\bar{\sigma}_{\Lambda \cup \partial \Lambda}=\bar{\omega}_{\Lambda \cup \partial \Lambda},
\eta_{\Lambda \cup \partial \Lambda}=\zeta_{\Lambda \cup \partial \Lambda})\cr
&=\frac{1}{Z_{\Lambda \cup \partial \Lambda}}
\bar{\mu}(\bar{\sigma}_{\Lambda \cup \partial \Lambda}=\bar{\omega}_{\Lambda \cup \partial \Lambda})
\prod_{(x,y) \in {}^\rho\Vec{E}_{\Lambda \cup \partial \Lambda}}\rho^Q(\zeta_{(x,y)} \mid \bar{\omega}_y-\bar{\omega}_x).
\end{split}
\end{equation}
In case of 2-height periodic GGMs, we also need to consider a different filtration compared to the case 
of free states which contains information on both the gradient field and the internal fuzzy chain.
\begin{defn}   
We let
\begin{equation*}
\overline{\mathcal{F}}^{\nabla}_{B_r}:=\sigma(\bar{\sigma}_x, \, \eta_b \mid x \in B_r, \, b \in \vec{E}_{B_r})
\end{equation*}
denote the sigma-algebra on $\Z_2^V \times\Omega^\nabla$ generated by the gradient field \textbf{and} the fuzzy chain inside $B_r$.
\end{defn}
\begin{pro}[Martingale construction for 2-height periodic GGM]\label{pro: MartingaleHiddenIsing}
Let $Q$ be a positive symmetric transfer operator with $\sum_{j=0}^\infty j^2 Q(j)<\infty$. Let $\bar{\mu}$ be any Gibbs measure for the Ising model described by the fuzzy transfer operator $\bar{Q}$. 

Then  the family $(H_r^{\rho,0})_{r \in \mathbb{N}}$ is an $\mathcal{L}^2(\bar{\nu}^{\bar{\mu}})$-martingale for the filtration $(\overline{\mathcal{F}}^{\nabla}_{B_r})_{r \in \mathbb{N}}$. In particular, it converges $\nu^{\bar{\mu}}$-almost surely and in $\mathcal{L}^2(\nu^{\bar{\mu}})$ to a limit $H^{\rho,0}$, which is $\nu^{\bar{\mu}}$-a.s. finite. 
\end{pro}
\begin{proof}[Proof of Proposition \ref{pro: MartingaleHiddenIsing}]
The family $(H_r^{\rho,0})_{r \in \mathbb{N}}$ is by construction already adapted to the coarser filtration $(\mathcal{F}_{B_r}^\nabla)_{r \in \N}$, hence it is also adapted to $(\overline{\mathcal{F}}^{\nabla}_{B_r})_{r \in \mathbb{N}}$.
By applying conditional independence of the children conditioned on their parent and the Markov property similarly to \eqref{eq: Martingale property}, the martingale property follows once we have shown that \begin{equation}\label{eq: 2perCondExp}
\bar{\nu}^{\bar{\mu}}(\sigma_v^{\rho,0}(\eta) \mid \bar{\sigma}_u, \, \sigma^{\rho,0}_u(\eta))=\sigma^{\rho,0}_u(\eta) \quad \bar{\nu}^{\bar{\mu}} \text{-a.s.}
\end{equation}
for all oriented edges $(u,v) \in \vec{E}$.
To prove \eqref{eq: 2perCondExp}, write
\begin{equation*}
\begin{split}
\bar{\nu}^{\bar{\mu}}(\sigma_v^{\rho,0}(\eta) \mid \bar{\sigma}_u, \, \sigma^{\rho,0}_u(\eta))&=\sigma^{\rho,s}_u(\eta)+\bar{\nu}^{\bar{\mu}}(\eta_{(u,v)} \mid \bar{\sigma}_u, \, \sigma^{\rho,0}_u(\eta))\cr 
&=\sigma_u^{\rho,0}(\eta)+\bar{\nu}^{\bar{\mu}}(\eta_{(u,v)} \mid \bar{\sigma}_u)  \quad \bar{\nu}^{\bar{\mu}} \text{-a.s.},
\end{split}
\end{equation*}
where the second equality follows from the fact that $(\bar{\sigma}_x)_{x \in V}$ is a Markov chain and $\eta$ is obtained by an edge-wise independent resampling given any realization of  $(\bar{\sigma}_x)_{x \in V}$.
By construction of the measure $\bar{\nu}^{\bar{\mu}}$ (recall \eqref{eq: 2perGGMProd}) we have that \begin{equation}\label{eq: plusminussymmetry}\bar{\nu}^{\bar{\mu}}(\eta_{(u,v)}=k \mid \bar{\sigma}_u=\bar{a})=\bar{\mu}(\bar{\sigma}_v=\bar{k}+\bar{a} \mid \bar{\sigma}_u=\bar{a})\frac{Q(k)}{\bar{Q}(\bar{k})}
 \text{ for all } k \in \Z \text{ and }\bar{a} \in \Z_2.   
\end{equation} As $Q$ is symmetric and $\overline{(-k)}=\bar{k}$ for every $k \in \Z$ this implies
\begin{equation}\label{eq: 2perCentred}
\bar{\nu}^{\bar{\mu}}(\eta_{(u,v)} \mid \bar{\sigma}_u=\bar{0})=\bar{\nu}^{\bar{\mu}}(\eta_{(u,v)} \mid \bar{\sigma}_u=\bar{1})=0.
\end{equation}
This proves \eqref{eq: 2perCondExp}.
Next, we will prove $\mathcal{L}^2$-boundedness of the martingale $(H_r)_{r \in \N}$. First, by the Assumption \ref{ass: SquareIntegrability} on $Q$ there is a $C<\infty$ such that for any edge $(u,v) \in \vec{E}$ and any $\bar{a} \in \Z_2$ we have
\begin{equation}\label{eq: UniformC}
\bar{\nu}^{\bar{\mu}}(\eta_{(u,v)}^2 \mid \bar{\sigma}_v-\bar{\sigma}_u=\bar{a})=\frac{\sum_{j \in \bar{a}+2\Z}j^2Q(j)}{\sum_{k \in \bar{a}+2\Z}Q(k)}<C.	
\end{equation}
Now, for all $r \in \N$ and $\bar{\mu}$-almost all $\bar{\omega} \in \Z_2^V$
\begin{equation}
\begin{split}
\bar{\nu}^{\bar{\mu}}((H_r^{\rho,0})^2 \mid \bar{\sigma}=\bar{\omega})&=\frac{1}{\vert W_r \vert^2}\bar{\nu}^{\bar{\mu}}\bigg( \bigl(\sum_{v \in W_r}\sum_{b \in \Gamma(\rho,v)}\eta_b \bigr)^2 \, \bigg \vert  \,\bar{\sigma}=\bar{\omega}\bigg) \cr&=\frac{1}{\vert W_r \vert^2}\sum_{v_1,v_2 \in W_r}\sum_{\genfrac{}{}{0pt}{}{b_1 \in \Gamma(\rho,v_1)}{b_2 \in \Gamma(\rho,v_2)}}\bar{\nu}^{\bar{\mu}}(\eta_{b_1}\eta_{b_2} \mid \bar{\sigma}=\bar{\omega})
\end{split}
\end{equation}
By \eqref{eq: 2perCentred}, the gradient variables $(\eta_b)_{b \in \vec{E}}$ are centered and by construction conditionally independent given any realization of the internal field $(\bar{\sigma}_x)_{x \in V}$. Hence, fix $v_1$ in $W_r$ and as in \eqref{eq: SplittingPointFree} let $u(v_2)$ denote the vertex at which the paths from $\rho$ to $v_1$ and from $\rho$ to $v_2$ split with $\mathbbm{u}(v_2)$ denoting the respective distance between $u(v_2)$ and the root to arrive at the following analogue of \eqref{eq: secMoment} 
\begin{equation}\label{eq: L2Z2Bound}
\begin{split}
\bar{\nu}^{\bar{\mu}}((H_r^{\rho,0})^2 \mid \bar{\sigma}=\bar{\omega})&=\frac{1}{\vert W_r \vert} \sum_{v_2 \in W_r}\sum_{b \in \Gamma(\rho,u(v_2))}\bar{\nu}^{\bar{\mu}}(\eta_{b}^2 \mid \bar{\sigma}=\bar{\omega})\cr
&\leq C \frac{1}{\vert W_r \vert}\sum_{n=1}^r n \, \vert \{v_2 \in W_r \mid \mathbbm{u}(v_2)=n\}\vert.
\end{split}
\end{equation}	
This is an upper bound in terms of the same
geometric quantities we have encountered 
before in the proof for the free state. 
By employing \eqref{eq: SplitPoint}, similarly to \eqref{eq: secMoment2} we obtain an upper bound on $\bar{\nu}^{\bar{\mu}}(H_r^{\rho,0})^2 \mid \bar{\sigma}=\bar{\omega})$ which is uniform in $r$ and $\bar{\omega}$. Hence, we have verified that we have an $\mathcal{L}^2({\bar{\nu}^{\bar{\mu}}})$-martingale. Hence, $(H_r)_{r \in \N}^{\rho,0}$ converges in $\mathcal{L}^2({\bar{\nu}^{\bar{\mu}}})$ and $\bar{\nu}^{\bar{\mu}}$-a.s. As $(H_r^{\rho,0})_{r \in \N}$ is already measurable with respect to the gradient variables alone, 
its almost sure limit must also be measurable w.r.t. the gradient variables alone
(and does not carry any dependence on the internal spins which were appearing the filtration).  
Hence the convergence results also hold for the annealed measure $\nu^{\bar{\mu}}$.
\end{proof}
\begin{rk}[Extensions of Propositions \ref{pro: MartingaleFree} and \ref{pro: MartingaleHiddenIsing}]\label{rk: ExtensionsPro12}\mbox{}\\
\begin{enumerate}[a)]
\item The statements of Propositions \ref{pro: MartingaleFree} and \ref{pro: MartingaleHiddenIsing} remain true if we replace the Cayley tree of order $d$ by the \textbf{regular $d$-ary} tree, which is the rooted tree where the root has $d$ nearest neighbors and every other vertex has $d+1$ nearest neighbors.  
\item The statements of Proposition \ref{pro: MartingaleHiddenIsing} remain true if we prescribe any \textit{mod}-2 fuzzy configuration $\bar{\omega} \in \{\bar{0},\bar{1}\}^{\mathbb{T}_d}$ and consider the conditional distribution $\bar{\nu}^{\bar{\mu}}(\cdot \mid \bar{\sigma}=\bar{\omega})$ of gradient spin variables instead of the unconditioned one. This follows from the symmetry \eqref{eq: 2perCentred} and the uniform bound \eqref{eq: UniformC}.
\end{enumerate}
\end{rk}
\begin{rk}[Concerning higher periods $q$]
In the case $q=2$ conditioning on the non-observable internal fuzzy spin preserves the $\pm$-symmetry (see \eqref{eq: 2perCentred}), and thus leads to the martingale property of $(H_r^{\rho,0})_{r \in \N}$ even conditional on the internal spins. This is in general no longer true for higher $q$.
\end{rk}
\subsection{Exponential upper bound on single-site marginals via moment generating function}
In the following, we give an exponential upper bound on the tail-distribution of the HOV $H$ for the free measure provided that the moment generating function of a single-gradient variable $\eta_b$:
    \[\hat{\psi}^{free}(t):=\mathbb{E}\big [\exp(t \eta_b ) \, \big ]=\sum_{j \in \Z}\exp(tj)\frac{Q(j)}{\Vert Q \Vert_1}\]
exists on an open interval $(-R,+R)$ around zero.
A similar upper bound also holds true for GGMs of height-period 2 when we replace $\hat{\psi}^{free}$ by $\hat{\psi}^{\bar{\mu}}_\star$ 
defined by 
\begin{equation}\label{eq: UB2perMomentGenerating}
\hat{\psi}^{\bar{\mu}}_\star(t):=\max_{\bar{a} \in \{ \bar{0},\bar{1}\}} \big \vert \, \bar{\nu}^{\bar{\mu}}(\exp (t \eta_{(v,w)}) \mid \bar{\sigma}_w-\bar{\sigma}_v=\bar{a}) \vert. 
\end{equation}
Note that the finiteness of $\hat{\psi}^{free}(t)$ for a given $t$ implies the finiteness of $\hat{\psi}^{\bar{\mu}}_\star(t)$. 
\begin{pro}[Exponential bound on single-site marginals]\label{thm: UB singlesitemarginals}
Let $\nu$ be the free measure (or a GGM of height-period $2$ associated with some fuzzy chain $\bar{\mu}$, resp.)
Assume that the moment generating function $\hat{\psi}^{free}$ (or the function $\hat{\psi}^{\bar{\mu}}_\star$ of \eqref{eq: UB2perMomentGenerating}, resp.) exists on the open interval $(-R,+R)$ with $R>0$. 

Then, for all $A<(d+1)R$ there exists a constant $B(A)<\infty$ such that 
\begin{equation}
\mu^H(\vert \sigma_\rho \vert 
 \geq s+1) \leq \nu(\vert H^{\rho,0} \vert \geq s ) \leq B(A)e^{-A s} \quad \text{for all } s\geq 0.
\end{equation}
with equality in case of $s \in \N$.
\end{pro}
For the proof of Proposition \ref{thm: UB singlesitemarginals} we will first relate the moment generating function $\hat{\varphi^{\nu}_r}$ of $H_r^{\rho,0}$, $ r \in \N$, defined by
\begin{equation}
\hat{\varphi}^{\nu}_r(t)=\mathbb{E}\left[\exp\left(\frac{t}{\vert W_r \vert} \sum_{v \in W_r} \sigma_v^{\rho,0} \right)\right], 
\end{equation}
where $\nu$ is the free-state (a GGM of height-period $2$, resp.)
to $\hat{\psi}^{free}$ (to $\hat{\psi}^{\bar{\mu}}_\star$ in case of 2-height periodic GGMs).
We will write $\hat{\varphi}_r^{free}$ when the expectation is taken with respect to the free measure and $\hat{\varphi}_r^{\bar{\mu}}$ when the expectation is taken with respect to $\nu^{\bar{\mu}}$.
\begin{lemma}[Relations for integral transforms]\label{pro: singleEdge vs Tree}
The following relation holds true for the moment generating function $\hat{\varphi}^{\nu}_r$ and any $t$ in the interval $(-R,R)$, where $\hat{\psi}^{free}(t)$ exists as a finite number.
\begin{enumerate}[a)]
    \item In case of $\nu$ being the free measure,
\[\hat{\varphi}^{free}_r(t)=\prod_{l=0}^{r-1} \hat{\psi}^{free}\left(\frac{t}{(d+1)d^{l}}\right)^{(d+1)d^l}.\]
    \item For the $\nu$ being a GGM of height-period $2$ with associated fuzzy chain $\bar{\mu}$, for real $t$ we have
    \[ \hat{\varphi}_r^{\bar{\mu}}(t)  \leq \prod_{l=0}^{r-1} \hat{\psi}_\star^{\bar{\mu}}\left(\frac{t}{(d+1)d^{l}}\right)^{(d+1)d^l}.\]
\end{enumerate}
Note that the Statements (a) and (b) remain true if we replace the moment generating functions by the respective characteristic functions, i.e., substitute $t=is$ for any $s \in \R$, and consider absolute values in (b).
\end{lemma}
\begin{proof}{Proof of Lemma \ref{pro: singleEdge vs Tree}.}
We start with considering the case when $\nu$ is the free measure. The proof is done by induction on $r \in \N$.
For $r=1$ we have \begin{equation}\label{eq: indBeg}
\hat{\varphi}_1^{free}(t)=\mathbb{E}[\exp(\frac{t}{d+1}\sum_{v \in W_1} \eta_{(\rho,v)})]=\hat{\psi}^{free}\left(\frac{t}{d+1} \right)^{d+1},    
\end{equation}
which follows from the fact that we condition the root to have height $0$ and independence of the gradient spin variables.
For the induction step write 
\begin{equation}\label{eq: recursionMomentfree}
\begin{split}
\hat{\varphi}_{r+1}^{free}(t)&=\mathbb{E}\bigg[\exp(\frac{t}{\vert W_{r+1} \vert} \sum_{v \in W_{r+1}} \sigma_v^{\rho,0})\bigg] = \mathbb{E}\bigg[\exp \bigg(\frac{t}{\vert W_{r+1} \vert} \sum_{v \in W_r} \sum_{w \in S(v)} (\sigma_v^{\rho,0}+\eta_{(v,w)})\bigg) \bigg].
\end{split}
\end{equation}
The gradient spin variables for the free measure are i.i.d., hence we arrive at
\begin{equation}
\begin{split}
\hat{\varphi}_{r+1}^{free}(t)&=\mathbb{E}[\exp (\frac{t}{d\vert W_{r} \vert} \sum_{v \in W_r} d \, \sigma_v^{\rho,0} )] \cdot \prod_{\genfrac{}{}{0pt}{}{v \in W_r}{w \in S(v)}}\mathbb{E}[\exp (\frac{t}{\vert W_{r+1} \vert}  \eta_{(v,w)}) ] \cr 
&=\hat{\varphi}_{r}^{free}(t) \cdot\hat{\psi}^{free}\left(\frac{t}{\vert W_{r+1} \vert}\right)^{\vert W_{r+1} \vert}=\hat{\varphi}_{r}^{free}(t) \cdot \hat{\psi}^{free}\left(\frac{t}{(d+1)d^{r}}\right)^{(d+1)d^{r}}    
\end{split}    
\end{equation}
This proves Statement a).

For the case of 2-height periodic GGMs, 
we will condition on the sigma-algebra
\begin{equation}
\overline{\mathcal{F}}_{B_r}:=\sigma(\bar{\sigma}_v \mid v \in B_r)    
\end{equation}
which is generated by the random environment, i.e., the fuzzy spins, within the ball of radius $r$ around the origin.
We show by induction on $r \geq 1$ that
\[ \big \vert \, \bar{\nu}^{\bar{\mu}}\big(\exp(\frac{t}{\vert W_{r} \vert} \sum_{v \in W_{r}} \sigma_v^{\rho,0}) \mid \overline{\mathcal{F}}_{B_{r+1}} \big) \, \big \vert\leq \prod_{l=0}^{r-1} \hat{\psi}_\star^{\bar{\mu}}\left(\frac{t}{(d+1)d^{l}}\right)^{(d+1)d^l} , \]
where we consider absolute values to extend the result to the case of the characteristic function.
Then Statement b) follows from total expectation, i.e., integration out w.r.t. hidden Ising measure $\bar{\mu}$ in combination with the triangle inequality for conditional expectations.
The induction beginning $r=1$ is completely analogous to \eqref{eq: indBeg} employing conditional independence of the gradient spin variables.
For the induction step we write
\begin{equation}\label{eq: treeRec2per}
\begin{split}
&\bigg \vert \, \bar{\nu}^{\bar{\mu}}\bigg(\exp \bigg(\frac{t}{\vert W_{r+1} \vert} \sum_{v \in W_r} \sum_{w \in S(v)} (\sigma_v^{\rho,0}+\eta_{(v,w)})\bigg) 
 \ \bigg \vert \ \overline{\mathcal{F}}_{B_{r+2}}\bigg) \, \bigg \vert \cr
&= \bigg\vert \, \bar{\nu}^{\bar{\mu}}\bigg(\exp \bigg(\frac{t}{d\vert W_{r} \vert} \sum_{v \in W_r} d \sigma_v^{\rho,0} \bigg) \bigg \vert \, \overline{\mathcal{F}}_{B_{r+2}}\bigg) \, \bigg\vert \cdot \prod_{\genfrac{}{}{0pt}{}{v \in W_r}{w \in S(v)}}\bigg\vert \bar{\nu}^{\bar{\mu}}\bigg(\exp \bigg(\frac{t}{\vert W_{r+1} \vert}  \eta_{(v,w)} \bigg) \, \bigg \vert \, \overline{\mathcal{F}}_{B_{r+2}} \bigg)\bigg\vert. 
\end{split}
\end{equation}
Hence, employing conditional independence of the gradient spins given the mod-2 increments along the respective edges 
we arrive at the upper bound $\prod_{l=0}^{r-1} \hat{\psi}_\star^{\bar{\mu}}\left(\frac{t}{(d+1)d^{l}}\right)^{(d+1)d^l}$ for   \eqref{eq: treeRec2per}.
This completes the induction step and thus completes the proof of Lemma \ref{pro: singleEdge vs Tree}.
\end{proof}
With Lemma \ref{pro: singleEdge vs Tree} in mind, we can now describe the domain of convergence of the moment generating function of the HOV $H$. 

\begin{lemma}[Range of convergence for the moment generating function of $H$ and its cumulants]\label{pro: RangeOfConvergence}\mbox{}\\
\begin{enumerate}[a)]
    \item Let $\nu$ be the free measure (a GGM of height-period $2$, resp.) Assume that $\hat{\psi}^{free}(t)$ ($\hat{\psi}^{\bar{\mu}}_\star(t)$, resp.)
converges for real $t \in (-R,+R)$. Then the following holds true. The moment generating function  $\hat{\varphi}_\infty^\nu$ of $H^{\rho,0}$ converges for real $t \in \left(-(d+1)R,+(d+1)R\right)$ and is given as the pointwise limit of the family $(\hat{\varphi}_r^\nu)_{r \geq 1}$. 
 \item In the specific case of the free state let $\kappa_n^{free,\mathbb{T}^d}$ denote the $n$th cumulant for $H^{\rho,0}$ and let $\kappa_n^{free}$ denote the $n$th cumulant associated with $\hat{\psi}^{free}$. \\Then 
$\kappa^{free,\mathbb{T}^d}_n = \kappa^{free}_n \frac{d^{n-1}}{(d^{n-1}-1)(d+1)^{n-1}}$.
\end{enumerate}
\end{lemma}
\begin{proof}
(a) First consider the case of the \textbf{free measure}. 
By definition of the cumulants $(\kappa_n^{free})_{n \in \N}$ of the increment along an edge we have
\[\log \psi^{free}(s) = \sum_{n=1}^\infty \kappa_n^{free} \frac{s^n}{n!}, \quad s \in (-R,+R).\]
Now Lemma \ref{pro: singleEdge vs Tree} gives for every $s \in (-(d+1)R,+(d+1)R)$:
\begin{equation}\label{eq: Cumulants}
\begin{split}
 &\limsup_{r \rightarrow \infty}\log \hat{\varphi}^{free}_r(s)\cr & = \limsup_{r \rightarrow \infty}\sum_{l=0}^{r-1}(d+1)d^l\log \psi^{free} \left(\frac{s}{(d+1)d^l} \right)=\limsup_{r \rightarrow \infty}\sum_{l=0}^{r-1}(d+1)d^l\sum_{n=1}^{\infty}\kappa^{free}_n \frac{s^n}{n!}\left(\frac{1}{(d+1)d^l} \right)^n \cr
 &=\limsup_{r \rightarrow \infty}\sum_{n=1}^\infty \kappa^{free}_n \frac{s^n}{n!}(d+1)^{-(n-1)}\sum_{l=0}^{r-1} d^{-(n-1)l}=\limsup_{r \rightarrow \infty}\sum_{n=1}^\infty \kappa^{free}_n \frac{s^n}{n!}\frac{1-d^{-(n-1)r}}{(d+1)^{n-1}(1-d^{-(n-1)})}\cr
 &=\sum_{n=1}^\infty \kappa^{free}_n \frac{s^n}{n!}\frac{1}{(d+1)^{n-1}(1-d^{-(n-1)})}.
\end{split}
\end{equation}
Here, interchanging the limsup w.r.t $r$ and the infinite summation w.r.t $n$ in the last equation is possible by dominated convergence. By the standard formula for the radius of convergence, the last expression in \eqref{eq: Cumulants} converges absolutely on the interval $(-(d+1)R,+(d+1)R)$. This also justifies the interchange of infinite summations in \eqref{eq: Cumulants}. By dominated convergence the moment generating function $\hat{\varphi}_\infty^\nu$ thus exists on $(-(d+1)R,+(d+1)R)$  and is given as the pointwise limit of $(\hat{\varphi}_r^\nu)_{r \geq 1}$, which proves (a) for the case of the free state. As a direct implication of  
\eqref{eq: Cumulants} we also obtain Statement (b).\\

\textbf{2-height periodic GGMs:} To adapt this proof to the case of 2-height periodic GGMs, define for any edge $(v,w) \in \vec{E}$ and real $t$ the expressions $ \hat{\psi}^{\bar{\mu}}_{\bar{0}}:=\bar{\nu}^{\bar{\mu}}(\exp (t \eta_{(v,w)}) \mid \bar{\sigma}_w-\bar{\sigma}_v=\bar{0})$ and $\hat{\psi}^{\bar{\mu}}_{\bar{1}}:=\bar{\nu}^{\bar{\mu}}(\exp (t \eta_{(v,w)}) \mid \bar{\sigma}_w-\bar{\sigma}_v=\bar{1})$, so $\hat{\psi}_\star(t)=\max_{\bar{a} \in \Z_2}\hat{\psi}_{\bar{a}}(t)$. Further, for $\bar{a} \in \Z_2$ let $\kappa_{n,\bar{a}}^{\bar{\mu}}$ denote the $n$th cumulant of the distribution of $\eta_{(v,w)}$ given $\bar{\sigma}_w-\bar{\sigma}_v=\bar{a}$. 
Hence, \[\log \hat{\psi}^{\bar{\mu}}_\star(t) \leq \vert \log \hat{\psi}^{\bar{\mu}}_{\bar{0}}(t)\vert+ \vert \log \hat{\psi}^{\bar{\mu}}_{\bar{1}}(t)\vert 
\leq \sum_{n=1}^{\infty}(\vert \kappa_{n,\bar{0}}^{\bar{\mu}}\vert+\vert \kappa_{n,\bar{1}}^{\bar{\mu}}\vert)\frac{\vert t \vert^n}{n!} \]
Then, by Lemma \ref{pro: singleEdge vs Tree} we have: 
\begin{equation}
\begin{split}
 &\limsup_{r \rightarrow \infty}\log \hat{\varphi}^{\bar{\mu}}_r(t) \leq \limsup_{r \rightarrow \infty}\sum_{l=0}^{r-1}(d+1)d^l\log \hat{\psi}_\star^{\bar{\mu}} \left(\frac{t}{(d+1)d^l} \right)\cr 
 &=\limsup_{r \rightarrow \infty}\sum_{l=0}^{r-1}(d+1)d^l\sum_{n=1}^{\infty}(\vert \kappa_{n,\bar{0}}^{\bar{\mu}}\vert+\vert\kappa_{n,\bar{1}}^{\bar{\mu}}\vert)\frac{\vert t \vert^n}{n!}\cr 
 &=\sum_{n=1}^\infty (\vert \kappa_{n,\bar{0}}^{\bar{\mu}}\vert+\vert\kappa_{n,\bar{1}}^{\bar{\mu}}\vert) \frac{\vert t \vert^n}{n!}\frac{1}{(d+1)^{n-1}(1-d^{-(n-1)})}
 \cr 
 &=\sum_{n=1}^\infty \vert \kappa_{n,\bar{0}}^{\bar{\mu}}\vert \frac{\vert t \vert^n}{n!}\frac{1}{(d+1)^{n-1}(1-d^{-(n-1)})}+\sum_{m=1}^\infty \vert\kappa_{m,\bar{1}}^{\bar{\mu}}\vert \frac{\vert t \vert^m}{m!}\frac{1}{(d+1)^{m-1}(1-d^{-(m-1)})}.
\end{split}
\end{equation}
From the assumption that $\hat{\psi}_\star(t)$ is finite for $t \in (-R,+R)$ we thus conclude similarly to the case of the free state that the pointwise limit $\hat{\varphi}^{\bar{\mu}}_r(t)$ exists as a finite number for all $t \in (-(d+1)R,(d+1)R)$. This concludes the proof of Statement (a) for the case of 2-height periodic GGMs.
\end{proof}
Finally, the proof of Proposition \ref{thm: UB singlesitemarginals} is just a combination of Lemma \ref{pro: RangeOfConvergence} and the (exponential) Markov inequality.
\begin{proof}[Proof of Proposition \ref{thm: UB singlesitemarginals}]
Let $0<A<(d+1)R$. Then we have for every $s \geq 0$
\begin{equation*}
\begin{split}
\nu(H^{\rho,0} \geq s)=\nu\left(\exp(AH^{\rho,0}) \geq \exp(As) \right) \leq \hat{\varphi}^{\nu}_\infty(A) \exp(-As),   
\end{split}    
\end{equation*}
where we employed Markov's inequality and the fact that by Lemma \ref{pro: RangeOfConvergence}, the number $A$ lies within the domain of convergence of $\hat{\varphi}^{\nu}_\infty$.
Similarly, $\nu(H^{\rho,0} \leq -s) \leq \hat{\varphi}^{\nu}_\infty(A) \exp(-As)$ as both for the free measure and GGM of height-period $2$, the distribution of $H^{\rho,0}$ is even.

Finally, the inequality $\mu^H(\vert \sigma_\rho \vert 
 \geq s+1) \leq \nu(\vert H^{\rho,0}) \vert \geq s )$ follows directly from the single-site marginals representation \eqref{eq: SingSiteMarginals}, which finishes the proof of Proposition \ref{thm: UB singlesitemarginals}.
 \end{proof}
 \subsection{The underlying HOV $H$ possesses an infinitely often differentiable Lebesgue-density}
 As a Corollary of Proposition \ref{pro: singleEdge vs Tree} applied to the characteristic functions we obtain the following result.
\begin{pro}[Exponential bound on limiting characteristic function]\label{pro: LebesgueCharacteristic}
Let $\nu$ either denote the free measure or a GGM of height-period $2$ associated with any fuzzy chain $\bar{\mu}$ and let  $\varphi_\infty^\nu$ denote the respective characteristic function of $H^{\rho,0}$. Then the following is true:
\begin{enumerate}[a)]
    \item There is a constant $0< c<\infty$ such that 
\[ \vert \varphi^\nu_\infty(t) \vert \leq \exp\left(-c \vert t \vert \right) \text{ for all } t \in \R. \]
\item By reverse Fourier-transform, $H^{\rho,0}$ thus has a Lebesgue density $f_{H^{\rho,0}}$ which is infinitely often differentiable.
\end{enumerate}
\end{pro}
\begin{proof}[Proof of Proposition \ref{pro: LebesgueCharacteristic}]
First assume that $\nu$ is the free state and denote by $\psi^{free}(s):=\hat{\psi}^{free}(is)$ and $\varphi^{free}_r(s):=\hat{\varphi}^{free}_r(is)$ the respective characteristic functions of a gradient spin variable and on the $r$-sphere.
To prove Statement (a), note that by a second-order Taylor expansion of $\psi^{free}$ (e.g., \cite[Thm. 15.32]{Kl20}) and vanishing firsts moments there are a constant $0<\bar{c}<\infty$ and $\delta=\delta(\bar{c})>0$ such that 
\begin{equation}
\log \psi^{free}(s) \leq -\bar{c} s^2 \text{ for every }s \in [0,\delta(\bar{c})]. \end{equation}
Consider the even function 
\[L_\delta: \R \rightarrow (0,\infty); \quad L_\delta(t):= \min \{l \in \N \mid \frac{\vert t \vert}{(d+1)d^l} \leq \delta \}. \]
Now, by Levy's continuity theorem (e.g., \cite[Thm. 15.24]{Kl20}) and Proposition \ref{pro: MartingaleFree} the family $(\varphi^{free}_r)_{r \in \N}$ converges pointwise (even more, uniformly on compact sets) to $\varphi^{free}_\infty$. Hence, from Lemma \ref{pro: singleEdge vs Tree} we obtain the following.
\begin{equation}
\begin{split}
\vert \varphi^{free}_\infty(t) \vert  &\leq \bigg | \prod_{l=1}^{L_\delta(t)}\psi^{free} \left(\frac{t}{(d+1)d^{l}} \right)^{(d+1)d^l} \bigg | \cr 
&\quad \cdot \, \bigg | \exp \left(\sum_{l=L_\delta(t)+1}^\infty \left((d+1)d^l \log\psi^{free} \left(\frac{t}{(d+1)d^{l}} \right)  \right) \right) \bigg |.
\end{split}
\end{equation}
As $\psi^{free}$ is a characteristic function, 
we can bound the absolute value of the first factor from above by $1$.
For the second factor we employ the defining properties of $\delta$ and $L_\delta$ to arrive at
\begin{equation}\label{eq: charExpUB}
\begin{split}
\vert \varphi^{free}_\infty(t) \vert &\leq \exp \left(-\bar{c}t^2\sum_{l=L_\delta(t)+1}^\infty (d+1)d^l \left( \frac{1}{(d+1)d^l}\right)^2\right) \cr 
&=\exp \left(-\frac{\bar{c}t^2}{(d+1)d^{L_\delta(t)+1}} \frac{d}{d-1} \right).
\end{split}
\end{equation}  
By definition of $L_\delta$, we have $\frac{\vert t \vert}{(d+1)d^{L_\delta(t)}} \geq \frac{\delta}{d}$. Hence, \eqref{eq: charExpUB} implies
\begin{equation}\label{eq: charExpUB2}
\vert \varphi^{free}_\infty(t) \vert \leq \exp \left(-\frac{\bar{c}\delta}{d(d-1)} \vert t  \vert \right).     
\end{equation}
In particular, the quantities $\bar{c}$ and $\delta(\bar{c})$ are positive and do not depend on $t$, which concludes the proof of Statement (a) for the free state setting $c:=\frac{\bar{c}\delta(\bar{c})}{d(d-1)}$.     
The respective result for any GGM of height-period $2$ now follows from replacing $\psi^{free}$ by $\psi^{\bar{\mu}}_\star$ defined by $\psi^{\bar{\mu}}_\star(s):=\hat{\psi}^{\bar{\mu}}_\star(is)$ (cf. \eqref{eq: UB2perMomentGenerating}) and noting that all the above steps remain true.

For the proof of Statement (b) let $\nu$ be either the free state or a GGM of height-period $2$. We note that Statement (a) implies that $\int_{-\infty}^{+\infty}\vert \varphi_\infty^\nu (t) \vert \, \text{d}t <\infty$.
Hence, from the Fourier-inversion formula (\cite[Chapter XV, Thm. 3]{Fe71} or \cite[Ex. 15.1.7]{Kl20}) we obtain that $H^{\rho,0}$ has the bounded continuous Lebesgue density
\begin{equation}
f_{H^{\rho,0}}(s)=\frac{1}{2\pi}\int_{-\infty}^{+\infty}\exp(-ist)\varphi_\infty^\nu(t) \, \text{d} t,    
\end{equation}
which by the exponential bound \eqref{eq: charExpUB2} and the differentiation lemma (e.g., \cite[Thm. 6.28]{Kl20}) is infinitely often differentiable.
\end{proof}
\begin{rk}[Extensions of Proposition \ref{pro: LebesgueCharacteristic}]\label{rk: ExtensionsPro5}
As for the results of Propositions \ref{pro: MartingaleFree} and \ref{pro: MartingaleHiddenIsing}, also the result of Proposition \ref{pro: LebesgueCharacteristic} remains true if we replace the Cayley tree of order $d$ by the respective regular $d$-ary tree. In the particular case of 2-height periodic GGMs the result further remains true if we fix 
any $\mod-2$ fuzzy spin configuration $\bar{\omega} \in \{\bar{0},\bar{1}\}^V$ and replace the distribution $\nu^{\bar{\mu}}$ of gradients by the conditional distribution $\bar{\nu}^{\bar{\mu}}(\cdot \mid \bar{\sigma}=\bar{\omega})$. This follows from the fact that $\psi^{\bar{\mu}}_\star$ is by definition an upper bound on the absolute value of the characteristic function conditional on the possible increments of the fuzzy chain. 
\end{rk}
\subsection{Pinning at infinity destroys spatial homogeneity}
Surprisingly, even pinning at infinity of a previously spatially homogeneous gradient Gibbs measure destroys spatial homogeneity already on the level of the single-site marginals.
We will provide a proof for the case of the free measure and afterwards outline its extension to the case of gradient Gibbs measures of height period $2$.
\begin{pro}\label{pro: InhomogeneousPinning}
Let $\nu$ be the free state (or a spatially homogeneous GGM of height-period $2$, resp.) and let $H$ be the height-offset variable which is given by Proposition \ref{pro: MartingaleFree} (or Proposition \ref{pro: MartingaleHiddenIsing}, resp.).

Then the Gibbs measure $\mu^H$ is not spatially homogeneous.
\end{pro}
\begin{proof}[Proof of Proposition \ref{pro: InhomogeneousPinning}]
Let $\mathbb{T}^d_n=(W_n,E_{W_n})$ denote the Cayley tree with root $\rho$ of depth $n$ and let $v \in \partial\{\rho\}$ be arbitrary. We will show that the joint distribution of the variables $\sigma_\rho$ and $\sigma_v$ is not invariant under the left translation with $v$,  \[l_v: V \rightarrow V, \quad u \mapsto vu,\] which maps $\rho$ to $v\rho=v$ and $v$ to $vv=\rho$.

Our proof strategy will be to 
exhibit a formula for their joint two-spin distribution, 
in which HOVs corresponding to regular $d$-ary trees will appear. 
While we only have partial knowledge about the distributional limits of 
those HOVs, 
this will in all cases 
be good enough for us to show that the resulting semi-explicit 
formula \eqref{eq: LimitOfDecomposition} can not possibly be invariant under the exchange of the 
vertices $v$ and $\rho$. 

For the particular choice of $v$ let $\{v_1,v_2,\ldots,v_d\}=\partial \{\rho \} \setminus \{v\}$. 
The choice of $v$ induces a decomposition of $\mathbb{T}^d_n$ into two subtrees $\overleftarrow{\mathbb{T}}^d_n$ and $\overrightarrow{\mathbb{T}}^d_n$. Here, $\overleftarrow{\mathbb{T}}^d_n$ is the regular $d$-ary tree of depth $n$ with root $\rho$ which lies in the past of the oriented edge $(\rho,v)$ (see Eq. \eqref{eq: the past}), while $\overrightarrow{\mathbb{T}}^d_n$ is the regular $d$-ary tree of depth $n-1$ rooted at $v$ which lies in the future of $(\rho,v)$. 
Now, recall the notation from Definition \ref{def: SphericalAverage} on spherical averages and write
\begin{equation}\label{eq: d-ary decomposition}
\begin{split}
H_r^{\rho,0}(\eta)=\frac{1}{\vert W_r \vert}\left(\sum_{u \in W_r \cap \overleftarrow{\mathbb{T}}^d_r}\sigma_u^{\rho,0}(\eta)+\sum_{u \in W_r \cap \overrightarrow{\mathbb{T}}^d_r}\sigma_u^{v,0}(\eta) \right)+\frac{\vert  W_r \cap \overrightarrow{\mathbb{T}}^d_r \vert}{\vert W_r \vert}\eta_{(\rho,v)}.
\end{split}
\end{equation}
In the next step notice that the statements of Propositions \ref{pro: MartingaleFree} and \ref{pro: MartingaleHiddenIsing} on almost sure convergence of  $(H_r^{\rho,0})_{r \geq 1}$ remain true if the underlying Cayley tree is replaced by the regular $d$-ary tree.
In particular, the $\nu$-almost sure limits
\begin{equation}\label{eq: LimitOfDecomposition1}
\overleftarrow{H}^{\rho,0}:=\lim_{r \rightarrow \infty}\frac{1}{\vert W_r \cap \overleftarrow{\mathbb{T}}^d_r \vert} \sum_{u \in W_r \cap \overleftarrow{\mathbb{T}}^d_r}\sigma_u^{\rho,0} \quad \text{and} \quad \overrightarrow{H}^{v,0}:=\lim_{r \rightarrow \infty}\frac{1}{\vert  W_r \cap \overrightarrow{\mathbb{T}}^d_r \vert} \sum_{u \in W_r \cap \overrightarrow{\mathbb{T}}^d_r}\sigma_u^{v,0} 
\end{equation}
exist.
See Figure \ref{Fig: SpliitingIntoSubtrees} for an illustration.
\begin{figure}[ht]
    \centering
    \includegraphics[width=0.5\linewidth]{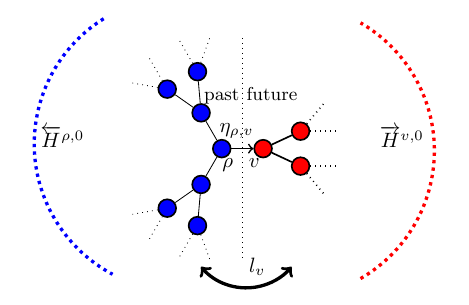}
    \caption{The situation on the binary tree $\mathbb{T}^2$. The choice of $v$ induces a splitting of $\mathbb{T}^d $ into the  blue regular $d$-ary subtree rooted at $\rho$ and the red regular $d$-ary subtree of rooted at $v$. The tail-measurable random variables 
$\ola{H}^{\rho,0}$ and $\ora{H}^{v,0}$ 
are the limits of spherical averages in the past (the future, resp.). Finally, the map $l_v$ interchanges the roles of the past and the future.}
\label{Fig: SpliitingIntoSubtrees}    
\end{figure}
Hence, taking the $\nu$-a.s.-limit $r \rightarrow \infty$ on both sides of the equation \eqref{eq: d-ary decomposition} leads to 
\begin{equation}\label{eq: LimitOfDecomposition}
H^{\rho,0}=\frac{d}{d+1}\overleftarrow{H}^{\rho,0}+\frac{1}{d+1}\overrightarrow{H}^{v,0}+\frac{1}{d+1}{\eta_{(\rho,v)}}
\end{equation}
Now consider the joint distribution of $(\sigma_\rho,\sigma_v)$ and of $(\sigma_v,\sigma_\rho)$. We have for any $s,t \in \Z$
\begin{equation}
\begin{split}
\mu^H(\sigma_\rho=t,\ \sigma_v=s)&=\mu^H \left(\sigma_\rho=t,\ \eta_{(\rho,v)}=s-t \right)\\ 
\mu^H(\sigma_v=t,\ \sigma_\rho=s)&=\mu^H \left(\sigma_v=t,\ \eta_{(\rho,v)}=-(s-t) \right). 
\end{split}
\end{equation}
From the expressions \eqref{eq: pinned measure muh} and \eqref{eq: GGMtoGM} describing the joint marginals of $\sigma_\rho$ and $\sigma_v$ under the pinned Gibbs measure $\mu^H$, together with the decomposition \eqref{eq: LimitOfDecomposition} of $H^{\rho,0}$ we obtain
\begin{equation}
\begin{split}
\mu^H(\sigma_\rho=t,\ \sigma_v=s)&=\nu\left(\lfloor \frac{d}{d+1}\overleftarrow{H}^{\rho,0}+\frac{1}{d+1}\overrightarrow{H}^{v,0}+\frac{s-t}{d+1} \rfloor=-t, \ \eta_{(\rho,v)}=s-t \right)\\ 
\mu^H(\sigma_v=t,\ \sigma_\rho=s)&=\nu\left(t-s-\lfloor \frac{d}{d+1}\overleftarrow{H}^{(\rho,0)}+\frac{1}{d+1}\overrightarrow{H}^{(v,0)}+\frac{t-s}{d+1} \rfloor=t, \ \eta_{(\rho,v)}=t-s \right).    
\end{split}    
\end{equation}
The claim will now follow 
if we compare these two formulas for general values 
of $s,t$ with $s-t\in (d+1)\Z$, to be detailed now.  

Set $Z:=\frac{d}{d+1}\overleftarrow{H}^{\rho,0}+\frac{1}{d+1}\overrightarrow{H}^{v,0}$ and for fixed $t,b \in \Z$ let $s:=s(b):=t+b(d+1)$. Then we have 
\begin{equation}
\begin{split}
\mu^H(\sigma_\rho=s,\ \sigma_v=t)&=\nu\left(\lfloor Z+b \rfloor=-t,\ \eta_{(\rho,v)}=b(d+1) \right)\\ 
\mu^H(\sigma_v=s,\ \sigma_\rho=t)&=\nu\left(-b(d+1)-\lfloor Z-b \rfloor=t, \ \eta_{(\rho,v)}=-b(d+1 )\right),     
\end{split}    
\end{equation}
which implies
\begin{equation}\label{eq: spatialHomPeriodicity2}
\begin{split}
\mu^H(\sigma_\rho=s,\ \sigma_v=t)&=\nu\left(\lfloor Z \rfloor=-t-b,\ \eta_{(\rho,v)}=b(d+1) \right)\\ 
\mu^H(\sigma_v=s,\ \sigma_\rho=t)&=\nu\left(\lfloor Z \rfloor=-t-bd, \ \eta_{(\rho,v)}=-b(d+1 )\right).    
\end{split}    
\end{equation}
\textbf{Free state:}
First consider the case of $\nu$ being the free state of i.i.d.-increments, where the random variable $Z$ is independent of $\eta_{(\rho,v)}$, because of prescribing the height $0$ at $\rho$ (at $v$, resp.) in \eqref{eq: LimitOfDecomposition1}. 
Now we will lead the assumption that $\mu^H(\sigma_\rho=s,\ \sigma_v=t)$ coincides with $\mu^H(\sigma_\rho=t,\ \sigma_v=s)$ for all $s,t \in \Z$ to a contradiction. \\Using independence of $Z$ and $\eta_{(\rho,v)}$ and symmetry of the distribution of $\eta_{(\rho,v)}$ the system \eqref{eq: spatialHomPeriodicity2} implies that 
\begin{equation}
 \nu\left(\lfloor Z \rfloor+t=-b \right)=\nu\left(\lfloor Z \rfloor+t=-bd\right)  \quad \text{for all } b \in \Z.  
\end{equation}
In particular, the sequence $(\nu\left(\lfloor Z \rfloor+t=-bd^n \right))_{d \in \N}$ is a constant, 
the only possibility being the constant 
$0$. As this has to hold for all $t$,
this contradicts the fact that the probability mass function of the discrete random variable 
$\lfloor Z \rfloor$ must sum to $1$. 
Hence, $\mu^H$ cannot be invariant under the translations of the tree.\medskip

\textbf{2-height periodic GGMs:}
In case of 2-height periodic GGMs we may further condition \eqref{eq: spatialHomPeriodicity2} on the $\mod-2$ spin values at $\rho$ and $v$ to obtain conditional independence between $Z$ and $\eta_{(\rho,v)}$ and afterwards integrate out the outer conditioning. More precisely, assume that $d$ is even (the case of $d$ odd is completely analogous by restricting to odd values of $a$) and restrict to even values of $b$. Then $b(d+1)$ is an even number, so in \eqref{eq: spatialHomPeriodicity2} the internal fuzzy chain must coincide at $\rho$ and $v$ to have an increment in class $\bar{0}$. Hence, assuming that $\mu^H(\sigma_\rho=s,\ \sigma_v=t)$ coincides with $\mu^H(\sigma_\rho=t,\ \sigma_v=s)$ for all $s,t \in \Z$ we obtain from \eqref{eq: spatialHomPeriodicity2} by  taking the marginals of the joint measure $\bar{\nu}^{\bar{\mu}}$ on the fuzzy spins and the integer-valued gradients
\begin{equation*}
\begin{split}
&\bar{\nu}^{\bar{\mu}}\left(\lfloor Z \rfloor=-t-b,\ \eta_{(\rho,v)}=b(d+1) \mid \bar{\sigma}_\rho=\bar{\sigma}_v=\bar{0}\right)\bar{\mu}(\bar{\sigma}_\rho=\bar{\sigma}_v=\bar{0}) \cr 
&\quad +\bar{\nu}^{\bar{\mu}}\left(\lfloor Z \rfloor=-t-b,\ \eta_{(\rho,v)}=b(d+1) \mid \bar{\sigma}_\rho=\bar{\sigma}_v=\bar{1}\right)\bar{\mu}(\bar{\sigma}_\rho=\bar{\sigma}_v=\bar{1}) \cr
&=\bar{\nu}^{\bar{\mu}}\left(\lfloor Z \rfloor=-t-bd,\ \eta_{(\rho,v)}=-b(d+1) \mid \bar{\sigma}_\rho=\bar{\sigma}_v=\bar{0}\right)\bar{\mu}(\bar{\sigma}_\rho=\bar{\sigma}_v=\bar{0}) \cr 
&\quad +\bar{\nu}^{\bar{\mu}}\left(\lfloor Z \rfloor=-t-bd,\ \eta_{(\rho,v)}=-b(d+1) \mid \bar{\sigma}_\rho=\bar{\sigma}_v=\bar{1}\right)\bar{\mu}(\bar{\sigma}_\rho=\bar{\sigma}_v=\bar{1}) \quad \text{ for all } b \in 2\Z.
\end{split}
\end{equation*}
Now, conditional independence of $Z$ and $\eta_{(\rho,v)}$, symmetry of the (conditional) distribution of $\eta_{(\rho,v)}$ and the fact that the conditional distribution of $\eta_{(\rho,v)}$ depends only on the height-increment of the fuzzy chain lead to 
\begin{equation*}
\begin{split}
&\bar{\nu}^{\bar{\mu}}\left(\lfloor Z \rfloor+t=-b \mid \bar{\sigma}_\rho=\bar{\sigma}_v=\bar{0}\right)\bar{\mu}(\bar{\sigma}_\rho=\bar{\sigma}_v=\bar{0}) \cr 
&\quad +\bar{\nu}^{\bar{\mu}}\left(\lfloor Z \rfloor+t=-b \mid \bar{\sigma}_\rho=\bar{\sigma}_v=\bar{1}\right)\bar{\mu}(\bar{\sigma}_\rho=\bar{\sigma}_v=\bar{1}) \cr
&=\bar{\nu}^{\bar{\mu}}\left(\lfloor Z \rfloor+t=-bd \mid \bar{\sigma}_\rho=\bar{\sigma}_v=\bar{0}\right)\bar{\mu}(\bar{\sigma}_\rho=\bar{\sigma}_v=\bar{0}) \cr 
&\quad +\bar{\nu}^{\bar{\mu}}\left(\lfloor Z \rfloor+t=-bd \mid \bar{\sigma}_\rho=\bar{\sigma}_v=\bar{1}\right)\bar{\mu}(\bar{\sigma}_\rho=\bar{\sigma}_v=\bar{1}) \quad \text{ for all } b \in 2\Z.
\end{split}    
\end{equation*}
As $d$ is even and $b$ runs through all even integers this again contradicts finiteness of the distribution of $\lfloor Z \rfloor$.

This concludes the proof of Proposition \ref{pro: InhomogeneousPinning}.
\end{proof}
\subsection{The distribution of $H$ is supported on $\R$.}
By extending the decomposition argument employed in the proof of Proposition \ref{pro: InhomogeneousPinning} above and combining it with the result of Proposition \ref{pro: LebesgueCharacteristic} on absolute continuity of the distribution of $H^{\rho,0}$ we prove the following result on the support of the distribution of $H^{\rho,0}$.
\begin{pro}\label{pro: SupportH}
Let $\nu$ be the free state (a spatially homogeneous GGM of height-period $2$ with a fuzzy Markov chain $\bar{\mu}$, resp.), $H$ be the height-offset variable which is given by Proposition \ref{pro: MartingaleFree} (Proposition \ref{pro: MartingaleHiddenIsing}, resp.) and $H^{\rho,0}$ the push-forward of $H$ to the gradient field under 
 pinning the height $0$ at $\rho$ as in Definition \ref{def: SphericalAverage}.
 
Then the smooth Lebesgue-density $f_{H^{\rho,0}}$ given by Proposition \ref{pro: LebesgueCharacteristic} is strictly positive on $\R$. 
\end{pro}
\begin{proof}[Proof of Proposition \ref{pro: SupportH}]
Fix any radius $r_1$, and write for $r$ sufficiently large the splitting into independent parts of the following form
\begin{equation}\label{eq: SplitBall}
\begin{split}
&
\frac{1}{\vert W_{r} \vert} \sum_{v \in W_{r}} \sigma_v^{\rho,0}=\frac{1}{\vert W_{r} \vert}\sum_{u \in W_{r_1}} \sum_{v \in W_{r}(u)} 
(\sigma_u^{\rho,0}+ \sigma_v^{u,0}) \cr
\end{split}
\end{equation}
where we write $W_{r}(u)$ for  the part of $W_r$ which can be reached by non-intersecting 
paths starting from $u$ pointing away from the origin. So, this is the part of $W_r$ which can be seen from $u$. In particular, $W_r(u)$ arises as the boundary of a regular $d$-ary subtree of depth $r-r_1$ rooted at $u$.
The last term can be rewritten as 
\begin{equation}
\begin{split}
&\frac{1}{\vert W_{r_1} \vert}\sum_{u \in W_{r_1}} \sigma_u^{\rho,0} 
+ \frac{1}{\vert W_{r_1} \vert}\sum_{u \in W_{r_1}} \Bigl( \frac{1}{\vert W_{r}(u) \vert}\sum_{v \in W_{r}(u)} \sigma_v^{u,0}\Bigr). 
\cr
\end{split}
\end{equation}
\textbf{Free measure:} First restrict to the case of $\nu$ being the free state. Then the random variables \[\Bigl( \frac{1}{\vert W_{r}(u) \vert}\sum_{v \in W_{r}(u)} \sigma_v^{u,0}\Bigr)_{u \in W_r}\] are independent. 
Again, apply Proposition \ref{pro: MartingaleHiddenIsing} to the regular $d$-ary tree rooted at $u$ (cf. Remark \ref{rk: ExtensionsPro12}) to conclude that as $r$ tends to infinity for fixed $r_1$, each of the independent $u$-indexed random variables 
$\frac{1}{\vert W_{r}(u) \vert}\sum_{v \in W_{r}(u)} \sigma_v^{u,0}$ converges to a HOV on the regular $d$-ary tree, anchored at $u$ which we denote by $\overrightarrow{H}^{u,0}$. 
Hence we have the equation 
\begin{equation}\label{eq: DiscVsCont}
\begin{split}
&H^{\rho,0}=\frac{1}{\vert W_{r_1} \vert}\sum_{u \in W_{r_1}} \sigma_u^{\rho,0}  + \frac{1}{\vert W_{r_1} \vert} \sum_{u \in W_{r_1}}  \overrightarrow{H}^{u,0}.
\end{split}
\end{equation}

Each of the random variables $\overrightarrow{H}^{u,0}$ under the right sum
has a density which is strictly positive on some small interval. It hence can be assumed to be larger than a fixed 
positive $\epsilon_0$ on a symmetric set of 
the form $J:=(\delta+(-\epsilon_1,\epsilon_1))\cup 
(-\delta+(-\epsilon_1,\epsilon_1))$. This is clear by the infinite differentiability of densities of HOVs we proved (extended to the regular $d$-ary tree). Note that positivity on $J$ implies that the density of the corresponding 
sum $\sum_{u \in W_{r_1}}  \overrightarrow{H}^{u,0}$ is strictly positive on 
the $|W_{r_1}|$-fold Minkowski sum of the set $J$ with itself, i.e. on the 
set $\{\sum_{u\in W_{r_1}}x_u \mid x_u \in J \text{ for all } u\}$.
Now observe that by choosing $r_1$ large enough, the Minkowski 
sum closes (possible) gaps and we obtain that the normalized sum
$ \frac{1}{\vert W_{r_1} \vert} \sum_{u \in W_{r_1}}  \overrightarrow{H}^{u,0}$ has density which is strictly 
bigger than some $\epsilon_2(r_1)$ for all $x$ in the whole interval $(-\delta,\delta)$. 
Note that $\epsilon_2(r_1)$ stays strictly positive for all sufficiently large but finite $r_1$. 

 On the other hand, the first term in \eqref{eq: DiscVsCont} is a discrete random variable with support $\frac{\Z}{\vert W_{r_1} \vert}$, which follows from the fact that the gradient variables are independent under the free measure and have range $\Z$ by the assumption that $Q(i)>0$ for 
all $i\in \Z$. 
Combining these two aspects we conclude that $f_{H^{\rho,0}}$ is positive on the set
$\frac{1}{\vert W_{r_1} \vert}\Z+
(-\delta,\delta)
$, for $r_1$ large enough. 
Choosing now $r_1$ large but finite proves the claim that the density is positive on the whole of $\R$. 
\medskip

\textbf{2-height periodic GGMs:}
To adapt this strategy of proof to 2-height periodic GGMs we first fix some internal finite-volume spin configuration $\bar{\omega}_{B_{r_1}}$ and condition on $\bar{\sigma}_{B_{r_1}}=\bar{\omega}_{B_{r_1}}$ to obtain an independent family \[\Bigl( \frac{1}{\vert W_{r}(u) \vert}\sum_{v \in W_{r}(u)} \sigma_v^{u,0,\bar{\omega}_u}\Bigr)_{u \in W_r}.\] Here,  we denote by $\sigma_v^{u,0,\bar{\omega}_{B_{r_1}}}$  the random variable that maps a gradient configuration drawn from the conditional distribution $\bar{\nu}^{\bar{\mu}}(\cdot \mid \bar{\sigma}_{B_{r_1}}=\bar{\omega}_{B_{r_1}})$ to the absolute height at $v$ obtained from prescribing the height $0$ at $u$.
Then, we make use that the statement of Proposition \ref{pro: MartingaleHiddenIsing} remains true if we replace the Cayley tree rooted at $\rho$ by the regular $d$-ary tree rooted at the respective $u$ \textbf{and} condition on $\bar{\sigma}_{B_{r_1}}=\bar{\omega}_{B_{r_1}}$ (cf. Remarks \ref{rk: ExtensionsPro12})
to obtain existence of the a.s. limits
\[\overrightarrow{H}^{u,0,\bar{\omega}_{B_{r_1}}}:=\lim_{r \rightarrow \infty}\frac{1}{\vert W_{r}(u) \vert}\sum_{v \in W_{r}(u)} \sigma_v^{u,0,\bar{\omega}_{B_{r_1}}}, \quad u \in W_{r_1}.\]
Let $H^{\rho,0,\bar{\omega}_{B_{r_1}}}$ denote the random variable constructed in Proposition \ref{pro: MartingaleHiddenIsing} with $\nu(\cdot)$ replaced by 
$\bar{\nu}^{\bar{\mu}}(\cdot \mid \bar{\sigma}_{B_{r_1}}=\bar{\omega}_{B_{r_1}})$ to arrive at the following analogue of \eqref{eq: DiscVsCont}:
\begin{equation}\label{eq: DiscVsCont2}
\begin{split}
&H^{\rho,0,\bar{\omega}_{B_{r_1}}}=\frac{1}{\vert W_{r_1} \vert}\sum_{u \in W_{r_1}} \sigma_u^{\rho,0,\bar{\omega}_{B_{r_1}}} + \frac{1}{\vert W_{r_1} \vert} \sum_{u \in W_{r_1}}  \overrightarrow{H}^{u,0,\bar{\omega}_{B_{r_1}}}.
\end{split}
\end{equation}
By construction of the 2-height periodic GGM, the two terms in \eqref{eq: DiscVsCont2} are independent under $\bar{\nu}^{\bar{\mu}}(\cdot \mid \bar{\sigma}_{B_{r_1}}=\bar{\omega}_{B_{r_1}})$. Furthermore, by the extension of Proposition \ref{pro: LebesgueCharacteristic} in the sense of Remark \ref{rk: ExtensionsPro5}, the second term and the l.h.s of \eqref{eq: DiscVsCont2} have infinitely-often differentiable Lebesgue-densities. Each of the random variables $\sigma_u^{\rho,0,\bar{\omega}_{B_{r_1}}}$ is the sum of independent gradient variables whose ranges are either all even integers or all odd integers (depending on the respective conditioning along the edge). 
For the particular case of the all-zero conditioning $\bar{\omega}_{B_{r_1}}=\bar{0}_{B_{r_1}}$ the range of the random variable $\frac{1}{\vert W_{r_1} \vert}\sum_{u \in W_{r_1}} \sigma_u^{\rho,0,\bar{\omega}_{B_{r_1}}}$ includes the $r_1$-dependent lattice $\frac{2}{\vert W_{r_1} \vert }\Z$.  
Writing the density as a sum over conditional densities in the form  
\begin{equation*}
f_{H^{\rho,0}}(\cdot)=\sum_{\bar{\omega}_{B_{r_1}} \in \Z_2^{B_{r_1}}}f_{H^{\rho,0,\bar{\omega}_{B_{r_1}}}}(\cdot) \,\, \bar{\mu}(\bar{\sigma}_{B_{r_1}}=\bar{\omega}_{B_{r_1}})
\end{equation*}
we can now proceed as in the case of the free measure.
\end{proof}
\subsection{The pinned measures are no Markov chains and not extreme}
A further inspection of the decomposition arguments of the proof of Proposition \ref{pro: SupportH} reveals the following:
\begin{pro}\label{pro: noMarkov}
\begin{enumerate}[a)]
    \item Let $\nu$ be the free state 
and let $H$ be the height-offset variable which is given by Proposition \ref{pro: MartingaleFree}. 
Then the Gibbs measure $\mu^H$ is not a tree-indexed Markov chain. In particular, it is not extreme in the set of all Gibbs measures.
\item  Let $\nu$ be a gradient Gibbs measure of height-period $2$ for some fuzzy chain $\bar{\mu}$ 
and let $H$ be the height-offset variable which is given by Proposition \ref{pro: MartingaleHiddenIsing}. 
Then the Gibbs measure $\mu^H$ is not extreme.
\end{enumerate}
\end{pro}
\begin{proof}[Proof of Proposition \ref{pro: noMarkov}]
First we prove Statement (a) regarding the free state.
The idea is to consider 
the two-point distribution of a spin 
at the root and at a neighbor $v$. Writing 
it in terms of ingredients coming from 
HOVs corresponding to the infinite past of the 
oriented edge from the root to $v$, 
and HOVs corresponding to the future, 
we will present a general argument showing 
that conditioning on the infinite past, 
necessarily has to have a non-trivial effect. That such an infinite memory remains 
is in particular incompatible with 
the tree-indexed Markov chain property. 
This is carried out as follows. 

Let $v \in \partial \{\rho \}$ and as in the proof of Proposition \ref{pro: InhomogeneousPinning} above consider the splitting
\begin{equation*}
H^{\rho,0}=\frac{d}{d+1}\overleftarrow{H}^{\rho,0}+\frac{1}{d+1}\overrightarrow{H}^{v,0}+\frac{1}{d+1}{\eta_{(\rho,v)}}
\end{equation*}
of the limit of averages over the $r$-spheres into the respective subgraphs of the past and of future (cf. Equation \eqref{eq: LimitOfDecomposition} and Figure \ref{Fig: SpliitingIntoSubtrees}).
Further abbreviate \[X:=\frac{d}{d+1}\overleftarrow{H}^{\rho,0}  \quad \text{and} \quad  Y:=\frac{1}{d+1}\overrightarrow{H}^{v,0}. \]
We will now draw the assumption that the associated $H$-pinned 
Gibbs measure $\mu^H$
is a Markov chain to a contradiction. 
More precisely, we will show that
\begin{equation}
\mu^H(\sigma_v=\cdot \mid \mathcal{F}_{(-\infty,\rho v]})= 
\mu^H(\sigma_v=\cdot \mid \sigma_{\rho}, \,X) \quad \mu^H-a.s.
\end{equation}necessarily depends in a non-trivial way on $X$.
In words, the distribution of the spin at $v$ recalls the part in the infinite past of the HOV $H$.
Note that the definition of a tree-indexed Markov chain involves a $\mu^H$-almost sure statement, hence we will take extra care to condition only on events with positive mass.
As we consider the case of $\nu$ being the free state of i.i.d.-increments, the gradient spin along the edge $(\rho,v)$ and the random variables $X$ and $Y$ are independent. 

Compute for any integers $t$ and $t'$ the conditional probability of $\sigma_v'=t'$ given $\sigma_\rho=t$ and $X \in A$, where $A$ is any Borel set with $\nu(X \in A)>0$.  
Then the single-site marginals representation \eqref{eq: SingSiteMarginals} and independence lead to \begin{equation}\label{eq: FactorizationOfDecomposition}
\begin{split}
&\mu^H(\sigma_v=t'\mid \sigma_{\rho}=t, X \in A)\cr
&=\frac{\nu\left(\, \lfloor {X + Y+ \frac{\eta_{(\rho,v)}}{d+1}}\rfloor=-t,\, 
\eta_{(\rho,v)}-\lfloor {X + Y + \frac{\eta_{(\rho,v)}}{d+1}}\rfloor=t'
 \, , X \in A\right)}{\nu\left(\, \lfloor {X + Y + \frac{\eta_{(\rho,v)}}{d+1}}\rfloor=-t  \, , X \in A\right)}\cr 
 &=\nu(\eta_{(\rho,v)}=t'-t) \cdot \frac{\nu\left( \,\lfloor {X + Y + \frac{t'-t}{d+1}}\rfloor=t, \ X \in A\right)}{\nu\left( \, \lfloor {X + Y + \frac{\eta_{(\rho,v)}}{d+1}}\rfloor=t, X \in A \right)}.
\end{split}
\end{equation}
Now assume that $\mu^H(\sigma_v=t'\mid \sigma_{\rho}=t, X \in A)$ would not depend on $A$. Then the second factor of the last expression in \eqref{eq: FactorizationOfDecomposition} would be some function $C(t,t')$ which does not depend on $A$. 
So, 
\begin{equation}\label{eq: tt'-constant}
\begin{split}
C(t,t')=\frac{\nu\left( \,\lfloor {X + Y + \frac{t'-t}{d+1}}\rfloor=t, \ X \in A\right)}{\nu\left( \, \lfloor {X + Y + \frac{\eta_{(\rho,v)}}{d+1}}\rfloor=t, X \in A \right)}
\quad \text{for all }t,t'\in \Z.
\end{split}    
\end{equation}
Now consider \eqref{eq: tt'-constant} for $t',t'' \in \Z$ and take the quotient of both expressions to obtain
\begin{equation}\label{eq: tt'-constant2}
\begin{split}
\tilde{C}(t,t',t''):=\frac{C(t,t')}{C(t,t'')}=\frac{\nu\left( \,\lfloor {X + Y + \frac{t'-t}{d+1}}\rfloor=t, \ X \in A\right)}{\nu\left( \,\lfloor {X + Y + \frac{t''-t}{d+1}}\rfloor=t, \ X \in A\right)}.    
\end{split}    
\end{equation}
By Proposition \ref{pro: SupportH} applied to the regular $d$-ary tree instead of the Cayley tree the Lebesgue-densities $f_{X}$ and $f_Y$ are strictly positive on $\R$. Take any $h \in \R$ and $\varepsilon>0$ specify to $A:=A_{h,\varepsilon}:=[h,h+\varepsilon]$ and further set $t':=t+d+1$ and $t'':=t$. Then employ the fact that $X$ and $Y$ have smooth Lebesgue densities $f_{X}$ and $f_{Y}$ to get from \eqref{eq: tt'-constant2} to
\begin{equation}\label{eq: tt'-constant3}
\begin{split}
 \hat{C}(t):=\tilde{C}(t,t+d+1,t)=\frac{\tfrac{1}{\varepsilon}\int_{A_{h,\varepsilon}} \lambda(dx)f_{X}(x)\int \lambda(dy)f_{Y}(y)1_{\{\lfloor x+y \rfloor=t-1 \}}}{\tfrac{1}{\varepsilon}\int_{A_{h,\varepsilon}} \lambda(dx)f_{X}(x)\int \lambda(dy)f_{Y}(y)1_{\{\lfloor x+y \rfloor=t \}}}.   
\end{split}    
\end{equation}
The assumption that $C(t,t')$ and hence $\hat{C}(t)$ do not depend on $A_{h,\varepsilon}$ now allows to consider the limit $\varepsilon \to 0$ on the right hand side, as we explain now.
By the continuity lemma (cf. \cite[Theorem 6.27]{Kl20}) the function \[A_{h,\varepsilon} \ni x \mapsto f_{X}(x)\int\lambda(dy) f_{Y}(y)1_{\{\lfloor x+y \rfloor=t-1 \}} \in \R\] is continuous and hence every $x \in A_{h,\varepsilon}$ is a Lebesgue point.
From the Lebesgue differentiation theorem we thus deduce \[\lim_{\varepsilon \rightarrow 0}\tfrac{1}{\varepsilon}\int_{A_{h,\varepsilon}}\lambda(dx)f_{X}(x)\int \lambda(dy)f_{Y}(y)1_{\{\lfloor x+y \rfloor=t-1 \}}=f_{X}(h)\int \lambda(dy)f_{Y}(y)1_{\{\lfloor h+y \rfloor=t-1 \}}.\]
As the density $f_{X}$ is positive on $\R$, we may apply the same reasoning to the denominator in \eqref{eq: tt'-constant3}. Thus we obtain
\begin{equation}\label{eq: tt'-constant4}
\hat{C}(t)=\frac{\nu(\lfloor Y+h \rfloor=t-1)}{\nu(\lfloor Y+h \rfloor=t)} \quad \text{ for all }h \in \R.    
\end{equation}
In particular, 
\begin{equation}
\hat{C}(t)=\frac{\nu(\lfloor Y-t \rfloor=l-1)}{\nu(\lfloor Y-t \rfloor=l)} \quad \text{ for all }l \in \Z.    
\end{equation}
This, is the desired contradiction to the fact that the probability mass function of the discrete random variable 
$\lfloor Y-t \rfloor$ must be summable.
Hence, we have proven that for the free state $\nu^{free}$ the pinned Gibbs measure $\mu^H$ cannot be a tree-indexed Markov chain.\medskip

Finally we prove statement (b), which says that a gradient Gibbs measure of height-period $2$ cannot be extreme in the set of Gibbs measures. This follows 
almost immediately from the non-degeneracy of the HOV we proved, 
when we recall the steps of its construction 
as a tail-measurable random variable and the construction of the HOV-pinned measure.
Following the framework of Sheffield (\cite[Lemma 8.4.2]{Sh05}) and its application to the graph-homomorphism model by Lammers and Toninelli \cite[Remark 2.7]{LaTo23} we define the height-offset spectrum of the HOV $H$ for a GGM $\nu$ as the distribution of $H$ w.r.t. to the Gibbs measure $\mu^H$.
By definition, $H$ is tail-measurable w.r.t. the absolute heights. Furthermore, by construction of $\mu^H$, the height-offset spectrum is described by the distribution of $H(\sigma^{x,s})- \lfloor H(\sigma^{x,s}) \rfloor$ under $\nu$ for any $x \in V$ and $s \in \Z$.
Now consider the set-up of Proposition  \ref{pro: MartingaleHiddenIsing} where $H$ is constructed as a limit of averages over spheres. Proposition \ref{pro: LebesgueCharacteristic} above guarantees that the distribution $H^{\rho,0}$ under $\nu$ is absolutely continuous w.r.t. Lebesgue measure.
In particular, the distribution of $H(\sigma^{x,s})- \lfloor H(\sigma^{x,s}) \rfloor$ is not a Dirac measure. Hence it provides
a non-trivial tail-measurable variable, and 
from the equivalence of extremality of Gibbs measures and their tail-triviality (\cite[Proposition 7.9]{Ge11}) we conclude that $\mu^H$ is not extreme in the set of Gibbs measures.
\end{proof}
\section{Applications}\label{sec: applications}
For the sake of illustration let us provide explicit expressions 
for the single-edge moment generating functions
in the cases of the following frequently studied examples of transfer operators describing the interactions of the model.
\subsection{SOS model}
The SOS model is described by the interaction $U(j)=\vert j \vert$, which corresponds to the transfer operator $Q$ defined by $Q_\beta^\text{SOS}(j)=\exp(-\beta \vert j \vert)$.
From this we obtain the single-gradient moment generating function
\begin{equation}\label{eq: LaplaceSOS}
 \hat{\psi}^{free}(s)=\frac{(\exp(\beta)-1)^2}{\exp(2\beta)-2\exp(\beta)\cosh(s)+1}, \quad s \in [0,\beta),   
\end{equation}
with a singularity at $s=\beta$.
Hence, all results of Theorems \ref{thm: Free measure} and \ref{thm: 2periodic} can be made explicit.  In particular we have the exponential concentration of the marginal at the root for the measure pinned at infinity, with any exponent smaller $\beta(d+1)$, which follows from the general concentration bounds of \hyperlink{res: localization free}{Theorem 2, b2)} and \hyperlink{res: localization 2per}{Theorem 3, b2)}.  
\subsection{Discrete Gaussian model}
The discrete Gaussian model is described by the interaction $U(j)= j^2 $, which corresponds to the transfer operator $Q$ defined by $Q_\beta^\text{Gau\ss}(j)=\exp(-\beta j^2)$.
From this we obtain the single-gradient Laplace-transform
\begin{equation}
 \hat{\psi}^{free}(s)=\exp\left(\frac{s^2}{2\beta} \right)\frac{\sum_{j \in \Z}\exp \left(-\frac{\beta}{2}(j-\frac{s}{\beta})^2 \right)}{\sum_{j \in \Z}\exp(-\frac{\beta}{2}j^2)}
\end{equation}
which gives us super-exponential concentration for the marginal. 
\section{Appendix}

\subsection{An example illustrating measurability properties w.r.t. boundary conditions}\label{App: measurability}

Following the suggestion of a referee we give explicit formulas for the gradient specification on a small volume, 
which in particular illustrates the relevant measurability properties 
w.r.t. to the boundary condition.

Let $d=1$, i.e. $V=\Z$. Take $\Lambda=\{0\}$, i.e., $\partial \Lambda =\{-1,+1\}$. 
To write out the corresponding gradient kernel it suffices to take $F$ to be a bounded gradient observable which depends 
only on the two gradient variables $\zeta_{(-1,0)},\zeta_{(0,1)}$. 
Then, expressing the gradient variables as differences of spin variables and summing over the free height variable $\omega_0\in \Z$ at the site $0$ we first obtain the explicit expression for the corresponding Gibbs kernel (see \eqref{eq: GibbsSpecification})
\begin{equation}
\begin{split}
&\gamma_{\{0\}}(
F(\omega_0-\omega_{-1},\omega_1-\omega_0)
 \mid \omega_{-1},\omega_{1}) \cr
&=\frac{\sum_{\omega_0\in \Z}
 F(\omega_{0}-\omega_{-1},\omega_{1}-\omega_0)
Q(\omega_{0}-\omega_{-1})Q(\omega_{1}-\omega_{0})
}{
\sum_{\omega_0\in \Z}
Q(\omega_{0}-\omega_{-1})Q(\omega_{1}-\omega_{0})}.
\end{split}
\end{equation}


Now note that the r.h.s. of the last equation depends on the variables 
$\omega_{-1},\omega_{1}\in \Z$ only through their difference 
$\omega_{1}-\omega_{-1}$, which can be seen by switching from the 
summation variable $\omega_{0}$ to $t=\omega_{0}-\omega_{1}$. 
Turning now to the gradient kernel and noting that $\omega_1-\omega_{-1}=\zeta_{(-1,0)}+\zeta_{(0,1)}$ determines the equivalence class $[\zeta]_{\{-1,1\}}$ we thus have verified that
\begin{equation}
\begin{split}
\gamma^\nabla_{\{0\}}\big(F(\zeta_{(-1,0)},\zeta_{(0,1)}) \mid \zeta \big) &= \gamma_{\{0\}}^\nabla\big(F(\zeta_{(-1,0)},\zeta_{(0,1)}) \mid [\zeta]_{\{-1,1\}}\big)\cr
&=\gamma_{\{0\}}\big(
F(\omega_0-\omega_{-1},\omega_0-\omega_1)
 \mid \omega_{-1},\omega_{1} \big),    
\end{split}    
\end{equation}
for any $\omega \in \Z^\Z$ such that $\omega_1-\omega_{-1}=\zeta_{(-1,0)}+\zeta_{(0,1)}$,
as demanded in \eqref{grad}. 
\subsection{Proof of Theorem \ref{thm: DLR}}
\begin{proof}[Proof of Theorem \ref{thm: DLR}]
Let $\emptyset \neq \Lambda \subset \Delta \Subset V$ and $\tilde{\omega} \in \Omega$ such that $\mu^H(\sigma_{\Delta \setminus \Lambda}=\tilde{\omega}_{\Delta \setminus \Lambda}) \neq 0$.  Further choose any $x \in \partial \Lambda$. Then, using Definition \ref{def: HeightOffset} and going over to gradients on the complement of $\partial \Lambda$ in the next step we have $\nu$-a.s.
\begin{equation}\label{eq: Anchoring At the Tail}
\begin{split}
&\mu^H(\sigma_{\Lambda \cup \partial \Lambda}=\tilde{\omega}_{\Lambda \cup \partial \Lambda} \mid \sigma_{\Delta \setminus \Lambda}=\tilde{\omega}_{\Delta \setminus \Lambda}) \cr&=\nu((\sigma^{x,0})_{\Lambda \cup \partial \Lambda} = \tilde{\omega}_{\Lambda \cup \partial \Lambda}+\lfloor H(\sigma^{x,0}) \rfloor \mid (\sigma^{x,0})_{\Delta \setminus \Lambda} = \tilde{\omega}_{\Delta \setminus \Lambda}+\lfloor H(\sigma^{x,0}) \rfloor)\cr
&=\nu(\eta_{\Lambda \cup \partial \Lambda} =(\nabla \tilde{\omega})_{\Lambda \cup \partial \Lambda}, \, (\sigma^{x,0})_{\partial \Lambda} = \tilde{\omega}_{\partial \Lambda}+\lfloor H(\sigma^{x,0}) \rfloor \cr 
&\qquad \mid \eta_{\Delta \setminus \Lambda} =(\nabla \tilde{\omega})_{\Delta \setminus \Lambda}, \, (\sigma^{x,0})_{\partial \Lambda} = \tilde{\omega}_{\partial \Lambda}+\lfloor H(\sigma^{x,0}) \rfloor) \cr
&=\nu(\eta_{\Lambda \cup \partial \Lambda} =(\nabla \tilde{\omega})_{\Lambda \cup \partial \Lambda} \mid \eta_{\Delta \setminus \Lambda} =(\nabla \tilde{\omega})_{\Delta \setminus \Lambda}, \, (\sigma^{x,0})_{\partial \Lambda}-\lfloor H(\sigma^{x,0}) \rfloor = \tilde{\omega}_{\partial \Lambda}).
\end{split}
\end{equation}
In the last line we used the elementary relation that the new probability of $A \cap B$ conditional on $B \cap C$ is the probability of $A$ conditional on $B \cap C$.\\
Note that the conditioning contains information on the gradient field outside $\Lambda$ as well as information on the relative heights at the boundary $\partial \Lambda$ together with the pinning prescribed by $H$. 
Hence, it is possible to equivalently describe the conditioning in terms of pinning at the point $x$ and the relative heights at the boundary.
More precisely, we have
\begin{equation}
\{(\sigma^{x,0})_{\partial \Lambda}-\lfloor H(\sigma^{x,0}) \rfloor = \tilde{\omega}_{\partial \Lambda}  \} = \{\lfloor H(\sigma^{x,0}) \rfloor  =-\tilde{\omega}_x  \} \cap \{[\eta]_{\partial \Lambda} = [\nabla \tilde{\omega}]_{\partial \Lambda} \} 
\end{equation}
By tail-measurability (with respect to the heights) of $H$ the event $\{\lfloor H(\sigma^{x,0}) \rfloor  =-\tilde{\omega}_x  \}$ is measurable with respect to $\mathcal{T}^\nabla_\Lambda$. Indeed, pinning the height $0$ at $x$ together with the relative heights at $\partial \Lambda$ and the gradient field outside $\Lambda$ provides all information on $(\sigma^{x,0})_{\Lambda^c}$.
Now employ the fact that the conditional distribution of the gradient field inside $\Lambda \cup \partial \Lambda$ under $\nu$ given $\mathcal{T}_\Lambda^\nabla$ is already measurable with respect to the relative heights at the boundary to conclude 
\begin{equation}
\begin{split}
&\nu(\eta_{\Lambda \cup \partial \Lambda} =(\nabla \tilde{\omega})_{\Lambda \cup \partial \Lambda} \mid \eta_{\Delta \setminus \Lambda} =(\nabla \tilde{\omega})_{\Delta \setminus \Lambda}, \, (\sigma^{x,0})_{\partial \Lambda}-\lfloor H(\sigma^{x,0}) \rfloor = \tilde{\omega}_{\partial \Lambda}) \cr
&=\nu(\eta_{\Lambda \cup \partial \Lambda} =(\nabla \tilde{\omega})_{\Lambda \cup \partial \Lambda} \mid [\eta]_{\partial \Lambda} = [\nabla \tilde{\omega}]_{\partial \Lambda})
\end{split}
\end{equation}
Taking into account the DLR-equation for the gradient Gibbs measure $\nu$ hence allows to rewrite the last expression of \eqref{eq: Anchoring At the Tail} as
\begin{equation*}
\begin{split}
&\nu(\eta_{\Lambda \cup \partial \Lambda} =(\nabla \tilde{\omega})_{\Lambda \cup \partial \Lambda} \mid \eta_{\Delta \setminus \Lambda} =(\nabla \tilde{\omega})_{\Delta \setminus \Lambda}, \, (\sigma^{x,0})_{\partial \Lambda}-\lfloor H(\sigma^{x,0}) \rfloor = \tilde{\omega}_{\partial \Lambda}) \cr
&= \gamma^\nabla(\eta_{\Lambda \cup \partial \Lambda} =(\nabla \tilde{\omega})_{\Lambda \cup \partial \Lambda} \mid \eta =\nabla \tilde{\omega}, \, [\eta]_{\partial \Lambda} = [\nabla \tilde{\omega}]_{\partial \Lambda})\cr
&=\gamma(\sigma_\Lambda= \tilde{\omega}_\Lambda \mid \tilde{\omega}),
\end{split}
\end{equation*}
which in combination with \ref{eq: Anchoring At the Tail} shows that $\mu$ satisfies the DLR-equation. This finishes the proof of Theorem \ref{thm: DLR}.
\end{proof}
\section*{Declarations}
\textbf{Data Availability:} Data sharing is not applicable to this article. The plots in Figures \ref{fig: Recursion} and \ref{Fig: SpliitingIntoSubtrees} have been generated within a separate .tex-file using the TikZ-package.

\textbf{Competing interests:} The authors have no competing interests to declare that are relevant to the content of this article.
\printbibliography

\end{document}